\documentclass[12pt]{article}
\usepackage{xcolor}
\usepackage{etoolbox}
\usepackage{thmtools}
\usepackage{amsmath,amsfonts,amssymb,amsthm}
\usepackage{ucs}
\usepackage{mathptmx}

\AtBeginDocument{
	\addtolength{\abovedisplayskip}{-2ex}
	\addtolength{\abovedisplayshortskip}{-2ex}
	\addtolength{\belowdisplayskip}{-2ex}
	\addtolength{\belowdisplayshortskip}{-2ex}
}
\usepackage[toc,page]{appendix}
\usepackage{setspace}
\usepackage{graphicx}
\graphicspath{ {./images/} }
\usepackage{wrapfig}
\usepackage[usenames,dvipsnames]{pstricks}
\usepackage{fancyhdr}
\usepackage{amsmath,amssymb,amsfonts}
\newtheoremstyle{break}
{}
{}
{\itshape}
{}
{\bfseries}
{\ }
{0pt}
{}

\theoremstyle{break}
\newtheorem{prop}{\textbf{Proposition}}[section]
\newtheorem*{theorem*}{Theorem}
\newtheorem{definition}{\textbf{Definition}}[section]
\newtheorem{theorem}{\textbf{Theorem}}[section]
\newtheorem{corollary}{\textbf{Corollary}}[section]

\newtheorem{claim}{\underline{\textbf{Claim}}}[section]

\newcommand{\R}{\mathbb{R}}

\newcommand{\N}{\mathbb{N}}

\newcommand{\abs}[1]{\lvert#1\rvert }

\graphicspath{ {./images/} }

\begin{document}
	\title{Sufficent Conditions for the preservation of Path-Connectedness in an arbitrary metric space}
	\author{S. Andronicou and E. Milakis}
	\date{}
	\maketitle
	\begin{abstract}
		It is proven that if $ (X,d) $ is an arbitrary metric space and $ U $ is a path-connected subset of $ X $ with $M:=\{x_i:\ i\in\{1,2,\dots,k\}\}\subset int(U) $, then the property of path-connectedness is also preserved in the resulting set $ U\setminus M, $ provided that the boundary of each open ball of X is a non-empty and path-connected set. Moreover, under appropriate conditions we extend the above result in the case where the set $ M $ is countably infinite. As a consequence these results maintain path-connectedness for domains with holes.
		
	\end{abstract}
	
	\textbf{Mathematics Subject Classification:} 54D05, 30L99
	
	\textbf{Keywords:} Path-connectedness, arbitrary metric spaces	
	 
	 \section{Introduction}
	The assumptions of connectedness and path-connectedness play a critical role when studying questions in Analysis, Topology and Partial Differential Equations. These conditions are present in a variety of theorems and applications, from regularity in free boundary problems, stability questions in fractured domains and derivation of interpolation inequalities to global bifurcation theory and the structure of solutions sets to vector equilibrium problems. For instance in the study of free boundaries in variational inequalities one seeks regularity of domains with finite perimeter in the sense of Caccioppoli and De-Giorgi (\cite{C52}, \cite{DG54}) while these domains verify the condition that the interior of its complement is composed by finite many connected components (see \cite{CR76}). In another occasion \cite{LL22}, the authors obtain geometric properties of nodal domains of solutions to certain uniformly elliptic equations. Again, the assumption of path-connectedness to these nodal domains leads to major impeccable results for instance when trying to obtain uniform bounds on frequency functions for solutions with same zero set. In general, there is a vast bibliography in these topics but we could, at least and additionally to the above, refer to  \cite{G07}, \cite{LM11}, \cite{LMS13}, \cite{YY23}, \cite{T21} and \cite{ZHW09}.
	
	 Having in mind the above mentioned topics, one ends up with an abstract question regarding path-connectedness.  
	 	 
	 \textit{Question: Assume that $ (X,d) $ is an arbitrary metric space and $ U $ is a path-connected subset of $ X $. If one removes a subset $M$ from $U$, is it true that resulting set $ U\setminus M $ becomes also path-connected?} 
	 
	 Although this question seems to have a fair answer for Euclidean settings, the case of arbitrary metric space appears a bit more tricky. To the best of our knowledge, the question remains open even in the case when $M$ is finite or countably infinite set of points. The main purpose of the present article is to fill exactly this gap, giving the proofs in full details. We prove that if $ (X,d) $ is an arbitrary metric space and $ U $ is a path-connected subset of $ X $ with $M:=\{x_i:\ i\in\{1,2,\dots,k\}\}\subset int(U) $, then the resulting set $ U\setminus M$ is also path-connected,  provided that the boundary of each open ball of X is a non-empty and path-connected set. Moreover, under appropriate conditions, the above result holds true in the case where the set $ M $ is countably infinite. Although this question has an abstract nature and is of independent interest across several topics of Mathematical Analysis, the motivation of the present paper is a bit more focused. In a forthcoming article \cite{AM24}, we are establishing properties of nodal domains for solutions to general elliptic type equations that include corkscrew conditions, Carleson type estimates and boundary Harnack inequalities. These estimates play a crucial role, for instance, in the study of two phase free boundary problems in which the positive and negative parts of a solution satisfy two (perhaps different) elliptic equations and the free boundary condition  involves normal derivatives from positive and negative sides. The proofs follow the fundamental ideas of \cite{LL22} and when one tries to reproduce the proof of Theorem 5.1 of \cite{LL22} the corresponding elements in the partition of the nodal domains need to stay path-connected, a condition which leads to the exact question we have imposed above.

	 The structure of the paper is as follows. In Section \ref{sec2} we provide the appropriate definitions, notations and present the main results while Section \ref{sec3} is devoted to the proofs of the main theorems.  
	
	
	\section{Definitions and Main Results}\label{sec2}

	\subsection{Definitions}
	In this section we give some necessary definitions and results for connected and path-connected metric spaces (for more details see \cite{BBT08}, \cite{E89}, \cite{F99}).
	\begin{definition}[Connected space] 
	A metric space\footnote{ In this paper, when we refer to the notion of a metric space $ (X,d) $, we mean a pair of objects $ (X,d), $ where the underlying set $ X $ is an an \emph{arbitrary} set (it could also be the empty set) and $ d $ is a map that fulfills the axioms of a metric. In bibliography, some authors exclude the empty metric space. We insist on using the more general definition. } $ \left( X,d\right)  $ is called \emph{connected} if there is no open partition i.e there is no pair of sets $ \{A,B\} $ such that the following holds:
		\begin{gather}
			(i)\ X=A\cup B,\ (ii) \ A,B\neq\emptyset,\ (iii)\  A\cap B=\emptyset,\ (iv)\ A, B\ \text{are open.}
		\end{gather}
		\end{definition}
	\begin{definition}[Path-connected space]
		A metric space $ \left( X,d\right)  $ is called \emph{path-connected} if for every $ x,y\in X,\ x\neq y $ there exists a continuous map $ \gamma:I\subset\R\to X $ where $ I $ is an interval, such that $ x,y\in\gamma(I). $
	\end{definition}
To be more specific, when we refer that a genuine subset $ A $ of $X$ fulfills the property of connectedness (path-connectedness), we mean that the metric subspace $ \left( A,d_A\right)  $ is connected (path-connected). We see that the notion of path-connectedness is stronger than then notion of connectedness. For example, the topologist's sine curve, i.e the set $ \{(0,0)\}\cup Gr(f) $ where $ f:(0,1]\to\R,\ f(x):=sin(1/x) $. In specific this set is connected in the eucledian space, but it is not path-connected, since there is no path connecting the origin to any other point on the space.\\
A well known result that relates continuity and connectedness is the following:
\begin{theorem}\label{continuity_connectedness}
	Let $ \left(X,d \right),\ \left(Y,\rho \right)   $ metric spaces, and $ f:X\to Y $ be a continuous function. If $ \left(X,d \right) $ is connected space, then the image $ f(X) $ is connected subset of $ Y $.
\end{theorem}
		Next we present our main results.
	\subsection{Main Results}
	\begin{theorem}\label{main_theorem_1} Let $ (X,d) $ be a metric space and $ U $ a path-connected subset of $ X $. We define $ M:=\{x_1,x_2,\dots x_k\}\subset int(U) $. If for every $ x\in X $ and $ \delta>0 $, the boundary of the open ball $ \partial B_{d}\left(x,\delta \right)  $ is a non-empty and path-connected set, then the set $ U\setminus M $ is also path-connected.
	\end{theorem}
In the case of violation of the assumption required in the above theorem, regarding the boundary of every open ball in the metric space to be path-connected, can ultimately lead to a violation of the result of the theorem. A trivial example that imposes the necessity of the assumption, is for the usual metric space  $ \left( \R,\rho_{\abs{\cdot}}\right).  $ In specific, in this space the boundary of every open ball $ B_{\rho_{\abs{\cdot}}}\left( x,r\right)=\left( x-r,x+r\right)  $ is the set $ \{x-r,x+r\}. $ Setting, now  $ U $ to be any bounded interval in $ \R $, (therefore $ U $ is path-connected) then for any  internal point $ x_0 $ of $ U $, we conclude that the set $ U\setminus\{x_0\} $ is not path-connected, since is not even connected set. Moreover, even if the condition of path-connectedness of the boundary of every open ball holds, the violation of $ M $ being a subset of the interior of $ U $, can ultimately lead to a violation of the result of the theorem. An example that illustrates the last case, is the metric space $ (\R^2,d_2) $ where in this space the boundary of every open ball is path-connected. Now, let $ y_0 $ be any fixed real number and let $ A $  be any finite subset of $ \R. $ Then, we set as $ U $ to be $ U:=\{(x,y_0):\ x\in\R\} $ and $ M:=\{(x,y_0):\ x\in A\} $. We observe that $ M\subset U $ but $ int(U)=\emptyset. $ Therefore $ M\nsubseteq U $. Easily we conclude that $ U $ is path-connected but $ U\setminus M  $ is not path-connected.

Next, we present a theorem that provides sufficent conditions for preservation of connectedness of a metric space after the extraction of a countably infinite subset. 

\begin{definition}[U-M property]
	Let $ (X,d) $ be a metric space, $ U $ a path-connected subset of $ X $ with $ M\subset U. $ We call that the set $ U $ satisfies the property $ (U-M) $ if and only if:\\
	$ \forall x,y\in U\setminus M,\ \forall\tilde{x}\in M, $ if $ \gamma:I\to U $ is a path that connects the points $ x,y $ and $ (x_j)_{j\in\N} $ a sequence of discrete points of $ M $ such that:
	\begin{gather}
		\exists z\in \{x_j:\ j\in\N\}\setminus\{\tilde{x}\},\ dist\left( \tilde{x},\{x_j:\ j\in\N\}\setminus\{\tilde{x}\}\right)  =d(\tilde{x},z)
	\end{gather}
then there exists continuous map $ \gamma^*_{z,\tilde{x}}:[t_z,t_{\tilde{x}}]\to U $ such that $ z=\gamma^*_{z,\tilde{x}}(t_z),\ \tilde{x}=\gamma^*_{z,\tilde{x}}(t_{\tilde{x}}) $ and $ \gamma^*_{z,\tilde{x}}\left((t_z,t_{\tilde{x}}) \right) \cap M=\emptyset. $ 
\end{definition}
\begin{theorem}\label{main_theorem_1_b}
	Let $ (X,d) $ be a metric space, $ U $ a path-connected subset of $ X $ with $ M\subset int(U) $ closed  countably infinite such that $ iso(M)=M $. If $ (U-M) $ property holds and for every $ x\in X $ and $ \delta>0,\ \partial B_{d}\left(x,\delta \right)  $ is a non-empty and path-connected set, then the set $ U\setminus M $ is also path-connected.

\end{theorem}
	
	\section*{Notations}
	Let $ (X,d) $ be a metric space and $ A\subset X. $\\
	$ B_d\left(x,\delta \right):=\{z\in X :\ d(z,x)<\delta\}  $ the usual open ball with center $ x\in X $ and radious $ \delta>0 $ \\
	$ int(A):=\{x\in X:\ \exists\delta>0,\ B_d(x,\delta)\subset A\} $, the internal of the set $ A $\\
	$ ext(A):=int(X\setminus A) $, the external of the set $ A $\\
	$ Cl(A)\equiv\bar{A}:=\{x\in X:\ x\ \text{limit point of the set A}\} $, the closure of the set $ A $\\
	$ \partial A:=\bar{A}\cap Cl\left( X\setminus A\right) $, the boundary of the set $ A $\\
	$ iso(A):=\{x\in X:\ \exists \delta_x>0,\ B_d(x,\delta_x)\cap A=\{x\}\} $, the set of isolated points of the set $ A $\\
	For a continuous function $ \gamma:I\to X $, where $ I $ is an interval in $ \R $,  we use the following symbolism:
	\begin{itemize}
		\item $ \gamma(I):=\{\gamma(t)| \ t\in I\} $ the trace of the curve $ \gamma $ on the interval $ I. $
		\item For $ x,y\in \gamma(I) $, let $ x=\gamma(t_x), y=\gamma(t_y) $ and without loss of generality, let $ t_x<t_y $, then we set: 
		$ \gamma_{xy}:=\{\gamma(t)|\ t\in[t_x,t_y]\} $, i.e the part of curve $ \gamma $ that connects the points $ x $ and $ y. $
	\end{itemize}

	
	\section{Proof of Main Results}\label{sec3}
	The proof of the above theorems is implemented through the following fundamental to our analysis proposition:
 
	\begin{prop}\label{main_prop}
		Let $ (X,d) $ be an arbitrary metric space, $ U\subset X,\ x_0\in int(U), $ $ z\in X\setminus\{x_0\}, $ such that $ \forall\delta>0,\ \partial B_d(x_0,\delta)\neq\emptyset. $ If $ \gamma:I\to X $ is a continuous function defined on the interval $ I\subset\R $ such that $ x_0=\gamma(t_{x_0}),z=\gamma(t_z), $ where $ t_{x_0},t_z\in I, $ then there exists $ \delta_0>0 $ such that
		\begin{gather}\label{eq_1}
				\emptyset\neq\gamma(\tilde{I}) \cap\partial B_d(x_0,\delta)\subset U,\ \forall\delta\in( 0,\delta_0] 
		\end{gather}
	where $ \tilde{I}:=\begin{cases}
		[t_{x_0},t_z],\ &t_{x_0}<t_z\nonumber\\
		[t_z,t_{x_0}],\ &t_z<t_{x_0}\nonumber.
	\end{cases} $
	\end{prop}
		\begin{center}
			\includegraphics[width=80mm]{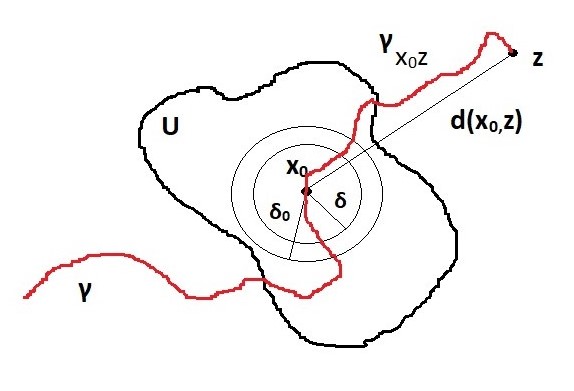}
		\end{center}
	\begin{proof}
		Without loss of generality, let $ t_{x_0}<t_z, $ i.e $ \tilde{I}:=[t_{x_0},t_z]. $ From the assumption that $ x_0\in int(U) $, there exists $ \epsilon_{x_0}>0 $ such that $ B_d\left(x_0,\epsilon_{x_0} \right)\subset U.  $ We define 
		\begin{gather}
			\delta_0:=\frac{1}{2}\min\{\epsilon_{x_0},d(z,x_0)\}.\nonumber
		\end{gather}
	We observe that $ \delta_0>0 $, since $ x_0\neq z. $ We will prove that $ \forall\delta\in(0,\delta_0],\ \gamma(\tilde{I})\cap\partial B_d\left(x_0,\delta \right)\neq\emptyset.  $\\
	Let suppose the opposite of what we seek to prove i.e 
	\begin{gather}\label{contradiction_assump}
		 \exists\tilde{\delta}\in(0,\delta_0],\ \gamma(I)\cap\partial B_d(x_0,\tilde{\delta})=\emptyset .
	\end{gather}
\textbf{Claim} $ z\in ext\left( B_d(x_0,\tilde{\delta}) \right) \equiv int\left( X\setminus B_d(x_0,\tilde{\delta})\right)  $.\\
In specific we prove that $ B_d(z,\tilde{\delta})\subset X\setminus B_d(x_0,\tilde{\delta}). $ Indeed, let $ y\in B_d(z,\tilde{\delta}),$ equivalently, $ d(y,z)<\tilde{\delta}. $ For what we are asking to prove, it is enough to show that $ y\notin B_d(x_0,\tilde{\delta}), $ equivalently $ d(y,x_0)\geq\tilde{\delta}. $ Let suppose that the last does not hold i.e $ d(y,x_0)<\tilde{\delta} $. Then,
\begin{gather}
	d(x_0,z)\leq d(x_0,y)+d(y,z)<2\tilde{\delta}\leq 2\delta_0\leq 2\frac{1}{2}d(x_0,z)=d(x_0,z)\nonumber.
\end{gather}
So we obtain $ d(x_0,z)<d(x_0,z) $ which is a contradiction.\\
We continue now with the fact that the metric space $ X $ receives a partition  in terms of the open ball $ B_d(x_0,\tilde{\delta}) $  $ \left( i.e\ int\left( B_d(x_0,\tilde{\delta})\right)=B_d(x_0,\tilde{\delta})\right)  $:
\begin{gather}
	X=B_d(x_0,\tilde{\delta}) \cup \partial B_d(x_0,\tilde{\delta})\cup ext\left( B_d(x_0,\tilde{\delta})\right) \nonumber.
\end{gather}
As a result from the above, we obtain
\begin{align}
	 \gamma(\tilde{I})&=\gamma(\tilde{I})\cap X \nonumber\\
	          &=\left(\gamma(\tilde{I})\cap B_d(x_0,\tilde{\delta}) \right) \cup \left(\gamma(\tilde{I})\cap \partial B_d(x_0,\tilde{\delta})  \right) \cup\left( \gamma(\tilde{I})\cap ext\left( B_d(x_0,\tilde{\delta}) \right)\right)\nonumber\\
	          &=^{(\ref{contradiction_assump})}\left(\gamma(\tilde{I})\cap B_d(x_0,\tilde{\delta}) \right)\cup\left( \gamma(\tilde{I})\cap ext\left( B_d(x_0,\tilde{\delta}) \right)\right)\nonumber\\
	          &=V_1\cup V_2\nonumber
\end{align}
where $ V_1:=\gamma(\tilde{I})\cap B_d(x_0,\tilde{\delta}) $ and $ V_2:=\gamma(\tilde{I})\cap ext\left( B_d(x_0,\tilde{\delta})\right) . $\\
First we see that from the Principle of Inheritance, the sets $ V_1, V_2 $ are open in the metric subspace $ \left(\gamma(\tilde{I}),d_{\gamma(\tilde{I})} \right)  $. Moreover, $ V_1\neq\emptyset, $ since $ x_0\in\gamma(\tilde{I})\cap B_d(x_0,\tilde{\delta})  $ and $ V_2\neq\emptyset, $ since from the previous claim, we have $ z\in ext\left(B_d(x_0,\tilde{\delta}) \right)  $ and $ z\in\gamma(\tilde{I}). $ Clearly, we see that $ V_1\cap V_2=\emptyset, $ since 
\begin{gather}
	B_d(x_0,\tilde{\delta})\cap ext\left( B_d(x_0,\delta)\right)= B_d(x_0,\tilde{\delta})\cap int\left(X\setminus B_d(x_0,\delta)\right)\subset B_d(x_0,\tilde{\delta})\cap \left(X\setminus B_d(x_0,\delta)\right)=\emptyset\nonumber.
\end{gather}
Consequently, the pair of sets $ \{V_1, V_2\} $ is an open (with respect to the relative metric $ d_{\gamma(\tilde{I})} $) partition of $ \gamma(\tilde{I}) $. From the last we conclude that the set $ \gamma(\tilde{I}) $ is not connected. The latter is contradicted with the fact that, since $ \tilde{I} $ is an interval in $ \R $, therefore a connected set and the function  $ \gamma:\tilde{I}\to X $, is continuous, it follows from Theorem \ref{continuity_connectedness} that $ \gamma(\tilde{I}) $ will be connected. As a result from the above analysis, we receive that:  
$  \forall\delta\in(0,\delta_0],\ \gamma(\tilde{I})\cap\partial B_d\left(x_0,\delta \right)\neq\emptyset.   $\\
 Moreover, since $ 0<\delta\leq\delta_0\leq\frac{1}{2}\epsilon_{x_0}<\epsilon_{x_0}, $ it follows that:
\begin{align}
	\gamma(\tilde{I})\cap\partial B_d\left(x_0,\delta \right)&\subset\gamma(\tilde{I})\cap\bar{B}_d\left(x_0,\delta \right)\nonumber\\
	&\subset\gamma(\tilde{I})\cap B_d\left(x_0,\epsilon_{x_0} \right)\nonumber\\
	&\subset ^{\left( B_d\left(x_0,\epsilon_{x_0} \right)\subset U\right) }\gamma(\tilde{I})\cap U\subset U\nonumber.
\end{align}
\end{proof}
 	
	It is worth noting that, in case where the metric space $ (X,d) $  has the property: $ \forall x\in X,\ \forall\delta>0,\ \partial B_d(x,\delta) $ is path-connected, then automatically (by the definition of path-connectedness) it is extracted that $ \partial B_d(x,\delta)\neq\emptyset,\ \forall x\in X,\ \forall\delta>0 $. Moreover, from the above proposition, an interesting result is derived regarding the minimum number of intersection points of the curve with the boundary of the ball.
	
	\begin{corollary}
		Let $ (X,d) $ is an arbitrary metric space, $ U\subset X, $   $ x_0\in int(U)$, $ x,y\in X\setminus\{x_0\} $ with $ x\neq y $ and $ \forall\delta>0,\ \partial B_d(x_0,\delta)\neq\emptyset.$ If $ \gamma:I\to X $ is a continuous function defined on the interval $ I\subset\R $, such that $ x,y,x_0\in \gamma(I) $ and $ \gamma_{xx_0}\cap \gamma_{x_0y}=\{x_0\}$, then there exists $ \tilde{\delta_0}>0 $ such that 
		\begin{gather}
			card\left( \gamma(I)\cap\partial B_d\left( x_0,\delta\right) \right)\geq 2,\ \forall\delta\in( 0,\tilde{\delta_0}] .
		\end{gather}
	\end{corollary}
	\begin{proof}
		Let $ x=\gamma(t_x),\ y=\gamma(t_y),\ x_0=\gamma(t_{x_0}). $ Since $ x\neq y $, it follows that $ t_x\neq t_y. $ Without loss of generality, let us suppose that $ t_x<t_y $. \\
		\textbf{Claim} $ t_x<t_{x_0}<t_y $\\
		Since $ x,y\in X\setminus\{x_0\} $, we obtain that $ t_x\neq t_{x_0} $ and $ t_y\neq t_{x_0}. $ Let us suppose that  $ t_x<t_y<t_{x_0} $. We see that
		\begin{gather}
			\gamma_{yx_0}:=\{\gamma(t)|\ t\in [t_y,t_{x_0}]\}\subset \{\gamma(t)|\ t\in [t_{x_0},t_x]\}:=\gamma_{x_0x}
		\end{gather}
	and from the above inclusion we also receive that $ \{x_0,y\}\subsetneq\gamma_{yx_0}\cap\gamma_{x_0x}, $ which contradicts with the assumption that $\gamma_{yx_0}\cap\gamma_{x_0x}=\{x_0\}.  $ We arrive in a similar way at a contradiction for the case $ t_{x_0}<t_x<t_y $. Consequently the above claim is true.\\
	Now, since $ x_0\in int(U) $, it exists $ \epsilon_{x_0}>0, $ such that $ B_d(x_0,\epsilon_{x_0}) \subset U.$ Next we define
	\begin{gather}
		\tilde{\delta}_0:=\frac{1}{2}\min\{\epsilon_{x_0},d(x,x_0),d(y,x_0)\}\nonumber.
	\end{gather}
It is clear that $ \tilde{\delta}_0>0 $ and 
\begin{gather}\label{eq_2}
	\tilde{\delta}_0\leq\delta^{(1)}_0:=\frac{1}{2}\min\{\epsilon_{x_0},d(x,x_0)\}\ \text{and}\
	\tilde{\delta}_0\leq\delta^{(1)}_0:=\frac{1}{2}\min\{\epsilon_{x_0},d(x,x_0)\}\nonumber.
\end{gather}
Then from the above relations $ (\ref{eq_2}), $ combined with the Proposition \ref{main_prop} we conclude that $ \forall\delta\in(0,\tilde{\delta}_0], $
\begin{gather}
	K_1^{(\delta)}:=\gamma_{xx_0}\cap\partial B_d(x_0,\delta)\neq\emptyset\ \text{and}\ 
	K_2^{(\delta)}:=\gamma_{x_0y}\cap\partial B_d(x_0,\delta)\neq\emptyset\nonumber
\end{gather}
	where $ \gamma_{xx_0}:=\{\gamma(t)|\ t\in [t_x,t_{x_0}]\} $ and $ \gamma_{x_0y}:=\{\gamma(t)|\ t\in [t_{x_0},t_y]\} $. Let $ z_1^{(\delta)}\in K_1^{(\delta)},z_2^{(\delta)}\in K_2^{(\delta)} $. It holds that $ K_1^{(\delta)}\cap K_2^{(\delta)}:=\emptyset. $ Indeed, if we supose that there exists $ z_{\delta}\in K_1^{(\delta)}\cap K_2^{(\delta)}, $ then $ z_{\delta}\in \gamma_{xx_0}\cap \gamma_{x_0y}.  $ By the assumption in the corollary, we conclude that $ z_{\delta}=x_0. $ Therefore, since $ z_{\delta}\in \partial B_d(x_0,\delta) $ we receive that $ x_0\in \partial B_d(x_0,\delta), $ which is a contradiction in an arbitrary metric space.\footnote{Let $ (X,d) $ be a metric space, $ x_0\in X $ and $ \delta>0. $ Then $  B_d(x_0,\delta)\cap\partial B_d(x_0,\delta)=\emptyset$. Let us suppose the opposite of what we seek to prove. I.e let $ z\in B_d(x_0,\delta)\cap\partial B_d(x_0,\delta)  $. We define $ \epsilon_z:=\delta-d(z,x_0) $. Notice that $ \epsilon_z>0 $ since $ z\in B_d(x_0,\delta).  $ Then, we see that $ B_d(z,\epsilon_z)\cap\left(X\setminus B_d(x_0,\delta) \right)=\emptyset,  $ since $ B_d(z,\epsilon_z)\subset B_d(x_0,\delta). $ This violates the definition of $ z $ being a limit point of $ X\setminus B_d(x_0,\delta). $ Therefore, $ z\notin Cl\left( X\setminus B_d(x_0,\delta) \right).   $ Consequently $ z\notin\partial B_d(x_0,\delta), $ which contradicts the initial assumption for $ z. $} Therefore, $ K_1^{(\delta)}\cap K_2^{(\delta)}=\emptyset $ resulting that $ z_1^{(\delta)}\neq z_2^{\delta}. $ In conclusion, $ z_1^{(\delta)},z_2^{(\delta)}\in\gamma(I)\cap\partial B_d(x_0,\delta), $ therefore $ card\left( \gamma(I)\cap\partial B_d\left( x_0,\delta\right) \right)\geq 2,\ \forall\delta\in( 0,\tilde{\delta}_0]. $
\end{proof}

	\subsection*{\textbf{Proof of Theorem \ref{main_theorem_1}}}
	\begin{proof}
			We begin the proof of the main result for the case $ k=1. $
		\newline
		\textbf{Case}: $ (k=1 )$\\
		In this case we prove that the set $ U\setminus\{x_0\} $ is path-connected. Let $ x, y\in U\setminus\{x_0\} $ two randomly selected points such that $ x\neq y. $ From the assumption of the theorem, $ U $ is path-connected, therefore there exists a continuous function $ \gamma:I\to U $, where $ I $ is an interval, such that $ x=\gamma(t_x) $ and $ y=\gamma(t_y), $ with $ t_x,t_y\in I. $  Since $ x\neq y $ it follows that $ t_x\neq t_y. $ Without loss of generality, we consider $ t_x<t_y. $ We distinguish the following cases:
		\begin{itemize}
			\item[(i)] $ x_0\notin\gamma([t_x,t_y]) $\\
			Therefore, $ \gamma([t_x,t_y])\subset U\setminus\{x_0\}$.
			\item[(ii)] $ x_0\in\gamma([t_x,t_y]) $\\
			Let $ t_{x_0}\in(t_x,t_y) $ such that $ x_0=\gamma(t_{x_0}). $ We define:
			\begin{gather}
				J:=\{t\in[t_x,t_y]|\ \gamma(t)=x_0\},\ \tilde{t}:=\inf J,\ t^*:=\sup J\nonumber.
			\end{gather}
			We observe that $ J\neq\emptyset, $ since $ t_{x_0}\in J. $ Moreover, $ J $ is lower and upper bounded by $ t_x $ and $ t_y $ respectively. Therefore, $ \tilde{t}, t^* $ are well defined. In specific, $ t_x\leq\tilde{t}\leq t^*\leq t_y. $ From the definition of  $ \tilde{t}, t^* $ we receive that both of them are limit points of $ J. $ Then by the sequential description of limit points, there are sequences $ (l_n)_{n\in\N},(c_n)_{n\in\N}\subset J $ such that $ l_n\xrightarrow{n\to\infty}\tilde{t} $ and $ c_n\xrightarrow{n\to\infty}t^*. $ Using now the continuity of $ \gamma, $ we receive that $ \gamma(l_n)\xrightarrow{d}\gamma(\tilde{t}) $ and $ \gamma(c_n)\xrightarrow{d}\gamma(t^*). $ Since $ \gamma(l_n)=\gamma(c_n)=x_0,\ \forall n\in\N $ by the uniqueness of limits for sequences, we finally obtain that $ \gamma(\tilde{t})=x_0=\gamma(t^*). $\\
			Next we claim that $ \tilde{t}>t_x, $ since if $ \tilde{t}=t_x, $ then $ \gamma(\tilde{t})=\gamma(t_x), $ i.e $ x_0=x, $ which contradicts with the assumption of the theorem. Similarly, we prove that $ t^*<t_y. $ Summarizing, we have that $ 	t_x<\tilde{t}\leq t^*<t_y $.\\
			From the assumption $ x_0\in int(U), $ there exists $ \epsilon_{x_0}>0 $ such that $ B_d(x_0,\epsilon_{x_0})\subset U. $ Next we define 
			\begin{gather}
				\delta_0=\frac{1}{2}\min\{\epsilon_{x_0},d(x_0,x), d(x_0,y)\}.
			\end{gather}
			Clearly we see that 
			\begin{gather}
				\delta_0\leq\delta_0^{(1)}:=\frac{1}{2}\min\{\epsilon_{x_0},d(x_0,x)\}\ \text{and}\ 
				\delta_0\leq\delta_0^{(2)}:=\frac{1}{2}\min\{\epsilon_{x_0},d(x_0,y)\}\nonumber.
			\end{gather}
			Using now the Proposition \ref{main_prop} we obtain that $ \forall\delta\in(0,\delta_0], $
			\begin{gather}
				\emptyset\neq K_1\equiv\gamma([t_x,\tilde{t}])\cap\partial B_d(x_0,\delta)\subset U\ \text{and}\
				\emptyset\neq K_2\equiv\gamma([t^*,t_y])\cap\partial B_d(x_0,\delta)\subset U\nonumber.
			\end{gather}
			Let $ z_1\in K_1 $ and $ z_2\in K_2 $. As a result from the last, 
			\begin{gather}
				\exists t_{z_1}\in [t_x,\tilde{t}\ ],\ \gamma(t_{z_1})=z_1\in\partial B_d(x_0,\delta)\ \text{and}\ 
				\exists t_{z_2}\in [t^*,t_y],\ \gamma(t_{z_2})=z_2\in\partial B_d(x_0,\delta)\nonumber.
			\end{gather}
			Initially, we claim that $ t_{z_1}<\tilde{t}\leq t^*<t_{z_2}. $ Indeed, let $ t_{z_1}=\tilde{t}. $ Then, $ \gamma(t_{z_1})=\gamma(\tilde{t}), $ i.e $ z_1=x_0 $ resulting that $ x_0\in\partial B_d(x_0,\delta), $ which is a contradiction. Consequently, $ t_{z_1}<\tilde{t}. $ Similarly, we prove that  $  t^*<t_{z_2}. $\\
			Next, using the assumption of theorem, the boundary of ball, $ \partial B_d(x_0,\delta) $ is path-connected set. Therefore, there exists a continuous function $ \tilde{\gamma}:[t_{z_1},t_{z_2}]\to\partial B_d(x_0,\delta) $ such that $ \tilde{\gamma}(t_{z_1}) =z_1$ and $\tilde{\gamma}(t_{z_2}) =z_2.  $ We notice that, 
			\begin{gather}
				\tilde{\gamma}([t_{z_1},t_{z_2}])\subset\partial B_d(x_0,\delta)\subset U\setminus\{x_0\}\nonumber.
			\end{gather}
			We define the following function $ \gamma^*:[t_x,t_y]\to U $ with
			\begin{equation}
				\gamma^*(t):=
				\begin{cases}
					\gamma(t) &  t \in [t_x,t_{z_1}]\\
					\tilde{\gamma}(t)& t\in[t_{z_1},t_{z_2}]\\
					\gamma(t) & t\in[t_{z_2},t_y].
				\end{cases}
			\end{equation}
			Spefically, the function $ \gamma^* $ is continuous and $ \gamma^*([t_x,t_y])\subset U\setminus\{x_0\}. $ Indeed, in order to prove the last, it is enough to valid the following:
			\begin{gather}\label{eq_3}
				\gamma^*\left([t_x,t_{z_1}]\cup[t_{z_2},t_y] \right) \subset U\setminus\{x_0\}.
			\end{gather}
			Let us suppose that there exists $ \bar{t}\in[t_x,t_{z_1}] $ such that $ \gamma^*(\bar{t})=\gamma(\bar{t}) =x_0.$ Then it follows that $ \bar{t}\in J. $ Therefore, $ \tilde{t}:=\inf J\leq \bar{t} $ which is a contradition, since $ \bar{t}\leq t_{z_1}<\tilde{t}. $ In a similar way, if we suppose the existence of an element $ \bar{t}\in[t_{z_2},t_y] $ such that $ \gamma^*(\bar{t})=\gamma(\bar{t}) =x_0,$ it follows that $ \bar{t}\in J. $ From the last we derive that $ t^*:=\sup J\geq\bar{t} $ which is a contradiction since $ \bar{t}\geq t_{z_2}>t^* $. Consequently, the inclusion $ (\ref{eq_3}) $ is true.
		\end{itemize}
		From the above two cases, we proved that for any two points $ x,y\in U\setminus\{x_0\} $ with $ x\neq y $ there exists a continuous function $ \gamma:I\subset\R\to U\setminus\{x_0\} $, where $ I $ is an interval, such that $ x,y\in\gamma(I). $ In consequence, the set $ U\setminus\{x_0\} $ is path-connected.\\
		We proceed now to the general case where $ k>1. $
		\newline
		\textbf{Case}: $ (k>1) $\\
		First, we define inductively the following finite sequence of sets $ (V_i)_{i\in\{1,2,\dots,k\}} $:
		\begin{gather}
			V_1:=U\setminus\{x_1\},\ \ V_i:= V_{i-1}\setminus\{x_i\},\ i=2,\dots,k.
		\end{gather}
		From the previous case $ k=1 $, since $ U $ is path-connected set and $ x_1\in int(U) $, we receive that $ V_1 $ is path-connected set. Let us suppose that $ V_i $ is path-connected for some random index $ i\in\{1,k-1\} $. Then, we notice that from the definition of $ V_{i+1} $,  the following holds: $ V_{i+1}:=V_{i}\setminus\{x_{i+1}\} $, where $ V_{i} $ is path-connected set (from the inductive step) and $ x_{i+1}\in int(V_i) $. Using now the previous case, we conclude that $ V_{i+1} $ is path-connected set. From the last finite inductive scheem, we receive that $ \forall i\in\{1,2,\dots,k\},\ V_i $ is path-connected set. Specifically, the set $ V_k=U\setminus\{x_1,x_2,\dots x_k\} $ is path-connected.
	\end{proof}

	
	Next we proceed to the proof of the Theorem \ref{main_theorem_1_b}.
	\subsection*{\textbf{Proof of Theorem \ref{main_theorem_1_b}}}
	\begin{proof}
		Let $ M $ be a closed, countable infinite subset of $ int(U) $ with all of its points to be isolated i.e $ iso(M)=M. $ We proceed to the proof that $ U\setminus M $ is path-connected. Let $ x,y\in U\setminus M $ two randomly selected points such that $ x\neq y. $ From the assumption of the theorem, $ U $ is path-connected, therefore there exists a continuous function $ \gamma:I\to U, $ where $ I $ is an interval such that $ x=\gamma(t_x) $ and $ y=\gamma(t_y) $ for some $ t_x,t_y\in I. $ Since $ x\neq y $ it follows that $ t_x\neq t_y. $ Without loss of generality, we consider $ t_x<t_y. $ Initially we set $S:=M\cap\gamma([t_x,t_y]).  $ We distinguish the following cases:\\
		\underline{\textbf{Case}}: ( $ S=\emptyset $)\\
			Therefore, we obtain $  \gamma([t_x,t_y])\subset U\setminus M  $.\\
		\underline{\textbf{Case}}: ($ S\neq\emptyset $ )\\
			Initially, we define
			 \begin{gather}
			 	t^*:=sup\{t\in[t_x,t_y]:\ \gamma(t)\in M\}.
			 \end{gather}
		 \begin{claim}\label{cl_0}
		 	The real number $ t^* $ is well defined and also satisfies the following:
		 	\begin{itemize}
		 		\item[(i)] $ \gamma(t^*)\in S\setminus\{x,y\} $, equivalently $ \gamma(t^*)\in M\cap\gamma\left(\left(t_x,t_y \right) \right)   $
		 		\item[(ii)] $ t_x<t^*<t_y $
		 		\item[(iii)] $ \gamma \left( (t^*,t_y)\right) \cap M=\emptyset $.
		 	\end{itemize}
		 \end{claim}
	 \begin{proof}(of Claim \ref{cl_0})
	 	First we prove that, $ t^* $ is well defined. Indeed, from the assumption of theorem, $ S\neq\emptyset$ i.e there exists $ z\in S, $ equivalently, there exists $ z\in M, $ such that $ z=\gamma(t_z) $ for $ t_z\in[t_x,t_y].  $ From the last, we conclude that $ z\in J:=\left\lbrace t\in[t_x,t_y]:\ \gamma(t)\in M \right\rbrace  $. Therefore, $ J\neq\emptyset. $ Obviously, $ J $ is upper bounded from $ t_y$  therefore $ t^* $ is well defined. Next, we proceed to the proof of $ (i)-(iii). $
	 	\begin{itemize}
	 		\item [(i)] Since $ t^*=\sup J $ is a limit point of $ J $ there exists $ (\tilde{t}_{l})_{l\in\N}\subset J $ such that $ \tilde{t}_l\xrightarrow{l\to\infty}t^*. $ By the continuity of $ \gamma$, it follows that $ \gamma(\tilde{t}_l)\xrightarrow{d}\gamma(t^*). $  Since $ (\tilde{t}_{l})_{l\in\N}\subset J $ we obtain that, $ \gamma(\tilde{t}_l)\in M $. Therefore $ \gamma(t^*) $ is a limit point of $ M. $ By the assumption $ M $ is closed, and as a result $ \gamma(t^*) \in M  $. Moreover, since  $\forall \in\N, \tilde{t}_l\in J\subset [t_x,t_y],  $ it follows that $\left( \gamma(\tilde{t}_l)\right) \subset \gamma\left([t_x,t_y] \right) .  $ From the last, we conclude also that $ \gamma(t^*) $ is limit point of $ \gamma\left([t_x,t_y] \right) . $ But the image $ \gamma\left([t_x,t_y] \right)  $ is closed as a continuous image of the compact set $ [t_x,t_y] $. As a result we receive finally that $ \gamma(t^*)\in \gamma\left([t_x,t_y] \right).  $ Now, if $ \gamma(t^*)=x=\gamma(t_x) $ then, due to the fact that $ \gamma(t^*)\in M, $ it follows then $ x\in M,  $ where the last contradicts with the assumption that $ x\notin M. $ By a similar argument, we prove that $ \gamma(t^*)\neq\gamma(t_y)=y. $ In conclusion, we receive that $ \gamma(t^*)\in M\cap \gamma\left(\left( t_x,t_y\right)  \right).  $
	 		\item [(ii)] First we see that $ t_x\in J $ and from the definition of $ t^* $ we obtain that $ t_x\leq t^*. $ Furthemore, $ t_y $ is an upper bound of $ J $, therefore $ t^*\leq t_y. $ If now $ t_x=t^* $ then $ x=\gamma(t_x)=\gamma(t^*) $ where from $ (i) $ we obtain that $ \gamma(t^*) \in M $, i.e $ x\in M. $ The last contradicts with our initial assumption, that $ x\in U\setminus M. $ As a result $ t_x<t^*. $ With a similar argument, we prove that $ t^*<t_y. $
	 		\item [(iii)] Let us suppose the opposite of what we week to prove i.e there exists $ z\in M $ such that $ z=\gamma(t_z) $ with $ t_z\in (t^*,t_y)\subset^{(ii)}[t_x,t_y]. $  It is obvious that $ t_z\in J. $ Therefore, $ t_z\leq \sup J=t^* $ which is a contradiction, since $ t_z>t^*. $
	 	\end{itemize}
	 \end{proof}
 Next we define inductively the following sequences $ (t_k)_{k\in\N},\ (t'_{k})_{k\in\N} $ of points in the interval $ [t_x,t_y]. $
 \begin{gather}
 	\begin{cases}
 		\text{We set}\  t_0:=t_x,\ t'_0:=t_x.\ \text{For}\ k\in\N\ \text{we define:}\nonumber\\
 		t_{k}:=\begin{cases}
 			\inf\left\lbrace t\in[t'_{k-1},t^*] :\ \gamma(t)\in M\setminus\left( \bigcup_{i=0}^{k-1}\left\lbrace  \gamma(t'_{i}) \right\rbrace\right) \right\rbrace , & \text{if }\ \gamma(t'_{k-1})\neq \gamma(t^*) \nonumber\\
 			t^*, & \text{if}\  \gamma(t'_{k-1})= \gamma(t^*)
 		\end{cases}
 	\\ t'_{k}:=\sup\left\lbrace t\in[t'_{k-1},t^*]:\ \gamma(t)=\gamma(t_k) \right\rbrace\nonumber.
 	\end{cases}
 \end{gather}
 
\begin{claim}[Properties of the sequences $ (t_k)_{k\in\N},\ (t'_k)_{k\in\N} $]\label{cl_1}
		The sequences $ (t_k)_{k\in\N},\ (t'_k)_{k\in\N} $ have the following properties:
		\begin{itemize}
			\item[(i)] They are well defined.
			\item[(ii)] $ \forall k\in\N,\ \gamma(t_k)=\gamma(t'_{k})\in M $.
			\item[(iii)] $ \forall k\in\N,\ t_x<t_k\leq t'_k\leq t^* $.
			\item[(iv)] If there exists $ k_0\in\N $ such that $t_{k_0}=t^*,\ (or\ t'_{k_0}=t^*) $ then $ \forall k\geq k_0,\ t_k=t^*=t'_k,\\ (resp. \forall k\geq k_0+1,\ t_k=t^*=t'_k ) $. I.e in this case the sequences  $ (t_k)_{k\in\N},\ (t'_k)_{k\in\N} $ are finally constant.
			\item[(v)] Let $ J:=\left\lbrace k\in\N:\ \gamma(t'_k) =\gamma(t^*) \right\rbrace $. \\ We define $ k^*:=\begin{cases}
				inf J, & \text{if}\ J\neq\emptyset\nonumber\\
				+\infty, & \text{if}\ J=\emptyset\nonumber
			\end{cases} $ and \ $ K:=\begin{cases}
			\{1,2,\dots,k^*\}, & \text{if}\ k^*\in\N\nonumber\\
			\N, & \text{if}\ k^*=\infty.
		\end{cases} $\\ Then, the following hold:
				\begin{itemize}
					\item[(a)]$ t_{k-1}<t^*,\ \text{and}\  t'_{k-1}<t^*, \ \forall k\in K$
					\item[(b)] $ \gamma(t_{k-1})=\gamma(t'_{k-1}) \neq\gamma(t^*),\ \forall k\in K$
					\item[(c)] $ t'_{k-1}<t_{k}, \ \forall k\in K $
					\item[(d)] $ t_{k-1}<t_k,  \ \forall k\in K $
					\item[(e)] $ t'_{k-1}<t'_k, \ \forall k\in K $
				\end{itemize} 
			\item [(vi)] $ \gamma(t_i)\neq\gamma(t_j),\ \forall i,j\in K,\ i\neq j $
			\item[(vii)] For all $ k\in K $, it holds that: $ \gamma\left(\left(t'_{k-1},t_{k} \right)  \right)\cap M=\emptyset $.
		\end{itemize}	 
\end{claim}
		 
		 \begin{proof}(of Claim \ref{cl_1})
		 	\begin{itemize}
		 		\item[(i)] We will prove by induction that the sequences $ (t_k)_{k\in\N},\ (t'_k)_{k\in\N} $ are well defined. First, for $ k=1, $ we observe from the Claim \ref{cl_0} (i), that $ \gamma(t'_0)=\gamma(t_x)=x\neq\gamma(t^*). $ Therefore, $ t_1=\inf J_1 $ where,  $ J_1=\left\lbrace t\in[t_x,t^*]:\ \gamma(t)\in M\setminus\{x\} \right\rbrace. $ Initially, we notice that $ J_1\neq\emptyset $, since $ t^*\in J_1 $ (see Claim \ref{cl_0} ) (i). Furtheremore, $ J_1 $ is lower bounded from $ t_x. $ Therefore $ t_1=\inf J_1 $ is well defined. Moreover, $ t'_1:=\sup J'_1 $ where $ J'_1:=\left\lbrace t\in[t_x,t^*]:\ \gamma(t)=\gamma(t_1)\right\rbrace.  $ It holds that $ t^* $ is an upper bound of $ J'_1 $ and also non empty set, since $ t_1\in J'_1. $ Indeed, from the fact that $ t_x $ is a lower bound of $ J_1, $ it follows that $ t_x\leq inf J_1 =t_1$ and since $ t^* $ is an upper bound of $ J_1 $ we obtain that $ t_1\leq t^* .$ In consequence, $ t_1\in[t_x,t^*]. $\\
		 		Let us suppose now that $ t_k $ and $ t'_k $ are well defined for random $ k\in\N $. We will prove  that $t_{k+1} $ and $ t'_{k+1}$  are well defined as well. We distinguish the following cases for $ \gamma(t'_{k}) $.
		 		\begin{itemize}
		 			\item Let $ \gamma(t'_{k})=\gamma(t^*) $. Then by the definition of $ t_{k+1},$ we receive that $ t_{k+1}=t^* $. Moreover, $ t'_{k+1}=\sup J'_{k+1} $ where $ J'_{k+1}:=\left\lbrace t\in[t'_k,t^*]:\ \gamma(t)=\gamma(t_{k+1})\right\rbrace=\left\lbrace t\in[t'_k,t^*]:\ \gamma(t)=\gamma(t^*)\right\rbrace  $. It is clear that $ J'_2\neq\emptyset, $ since $ t^*\in J'_{k+1} $ and also $ J'_{k+1} $ is upper bounded from again from $ t^*. $ In consequence, $ t'_{k+1} $ is well defined.
 		 			\item  Let $ \gamma(t'_{k})\neq\gamma(t^*) $. Then by the definition of $ t_{k+1},$ we receive that $ t_{k+1}=\inf J_{k+1}$, where $ J_{k+1}:=\left\lbrace t\in[t'_{k},t^*]:\ \gamma(t)\in M\setminus\left( \cup_{i=0}^{k}\left\lbrace \gamma(t'_i)\right\rbrace \right)  \right\rbrace.  $
		 			Initially, we observe that $ J_{k+1} $ is lower bounded from $ t'_k. $ Furtheromore, we claim that $ J_{k+1}\neq\emptyset. $ In specific, we claim that $ t^*\in J_{k+1}. $ First, we notice from Claim \ref{cl_0} (i), that $ \gamma(t^*)\in M $. Moreover, from the assumption, we receive that $ \gamma(t^*)\neq\gamma(t'_k) $ and $ \gamma(t'_{0}) =\gamma(t_x)=x\neq\gamma(t^*).$ (see again Claim \ref{cl_0} (i)). If now $ k\geq 2 $, let us suppose that there exists $ i_0\in\{1,2,\dots,k-1\},\ \gamma(t'_{i_0})=\gamma(t^*). $ Then, from the definition of the term $ t_{i_0+1} $, we obtain that $ t_{i_0+1}=t^* $ and $ t'_{i_0+1}=t^*, $ which results that, $ \gamma(t'_{i_0+1})=\gamma(t^*). $ Again, from the definition of the term, $ t_{i_0+2}, $ we obtain that $ t_{i_0+2}=t^* $ and $ t'_{i_0+2}=t^*. $ Continuing this procedure for  finite steps until we reach the term $ t_{k-1}, $ and conclude that $ t_{k-1}=t^*=t'_{k-1}. $ Therefore, $ \gamma(t'_{k-1})=\gamma(t^*) $ and as a consequence from the definition of $ t_k, $ we receive that $ t_k=t^*=t'_k=t^*. $ This leads to the fact that $ \gamma(t'_k)=\gamma(t^*) $ which contradicts our initial assumption. Therefore, $ \gamma(t^*)\in M\setminus\left(\cup_{i=0}^k \{\gamma(t'_i)\}\right)  $. I.e $ t^*\in J_{k+1}. $ From the last we conclude that, $ t_{k+1}=\inf J_{k+1} $ is well defined. Next, we have that, $ t'_{k+1}=\sup J'_{k+1} $ where $ J'_{k+1}=\left\lbrace t\in [t'_k,t^*]:\ \gamma(t)=\gamma(t_{k+1}) \right\rbrace  $. Obviously, $ J'_{k+1} $ is upper bounded from $ t^*. $ Moreover, $ J'_{k+1}\neq\emptyset,  $ since $ t_{k+1}\in J'_{k+1}. $ Indeed, since $ t'_k $ is a lower bound of $ J_{k+1}, $ it follows that $ t_{k+1}=\inf J_{k+1}\geq t'_k. $ Now, due to the fact, that $ t^* $ is an upper bound of $ J_{k+1}, $ it follows that $ t_{k+1}\leq t^*. $ In conclution, $ t_{k+1}\in [t'_k,t^*]. $ From the above analysis, through the Induction Principle, we receive the desired result. 
		 		\end{itemize}
		 		
		 		\item[(ii)] Let $ k_0\in\N $ randomly selected. We distinguish the following cases for $ \gamma(t'_{k_0-1})   $: Let $ \gamma(t'_{k_0-1}) =\gamma(t^*)$ then $ t_{k_0}=t^* $ and $ \gamma(t_{k_0})=\gamma(t^*) $. From the Claim \ref{cl_0} (i)  it follows $ \gamma(t^*) \in M$, consequently $ \gamma(t_{k_0})\in M $. Now let $ \gamma(t'_{k_0-1})\neq\gamma (t^*), $ then $ t_{k_0}=\inf J_{k_0} $.  From the fact that, $ t_{k_0}:=\inf J_{k_0} $ is a limit point of $ J_{k_0}, $ there exists a sequence $ (\tilde{t}_l)_{l\in\N}\subset J_{k_0} $ such that $ \tilde{t}_l\xrightarrow{l\to\infty}\inf J_{k_0}\equiv t_{k_0}. $ Consequently, from the continuity of $ \gamma $ we receive that, $ \gamma(\tilde{t}_l)\xrightarrow{l\to\infty}\gamma(t_{k_0})$. Since, $ (\tilde{t}_l)_{l\in\N}\subset J_{k_0}, $ it follows that $ \gamma(\tilde{t}_l)\in M,\ \forall l\in\N. $ Therefore, $ \gamma(t_{k_0})\in M $ since $ M $ is closed. Next, we prove that, $ \gamma(t'_{k_0})=\gamma(t_{k_0})\in M. $ Since $ t'_{k_0}=\sup J'_{k_0} $ and $ \sup J'_{k_0}  $ is a limit point of $ J'_{k_0} $ there exists a sequence $ \left(\bar{t}_l \right)_{l\in\N}\subset J'_{k_0}  $ such that $ \bar{t}_l\xrightarrow{l\to\infty}\sup J'_{k_0}\equiv t'_{k_0} $. By the contuinity of $ \gamma $, we receive $ \gamma(\bar{t}_l)\xrightarrow{l\to\infty}\gamma(t'_{k_0}) $. From the fact that $\left(\bar{t}_l \right)_{l\in\N}\subset J'_{k_0}$ we have that $ \gamma(\bar{t}_l)=\gamma(t_{k_0}),\ \forall l\in\N.  $ By the uniqueness of limit, it follows that $ \gamma(t'_{k_0}) =\gamma(t_{k_0})\in M. $ 
		 		
		 		\item[(iii)] Let $ k_0\in\N $ randomly selected. From the $ (ii) $, we conclude that $ t_x<t_{k_0} $, since if $ t_x=t_{k_0} $, then it follows that $ x=\gamma(t_x)=\gamma(t_{k_0}) $ i.e $ x\in M $ which contradicts with the assumption that $ x\in U\setminus M. $ By the definition of $ t'_{k_0} $ we obtain that, $ t'_{k_0}\leq t^*, $ since $ t^* $ is an upper bound of $ J'_{k_0}. $ Next, we prove that $ t_{k_0}\leq t'_{k_0}. $ Let $ \gamma(t'_{k_0-1})=\gamma(t^*)  $, then $ t_{k_0}=t^* $ and $ \gamma(t_{k_0}) =\gamma(t^*)$. Also from the last we obtain that $ t^*\in J'_{k_0} $. Consequently, since $ t^* $ is an upper bound of $ J'_{k_0} $, we receive $ t'_{k_0}=\sup J'_{k_0}=t^*$  and $ t_{k_0}=t'_{k_0}=t^*. $ Let now $ \gamma(t'_{k_0-1} )\neq\gamma(t^*)  $, then $ t_{k_0}=\inf J_{k_0}$. It holds that, $ t_{k_0}\geq t'_{k_0-1} $ since $ t'_{k_0-1} $ is a lower bound of $ J_{k_0}. $ Moreover, $ t^* $ is an upper bound of $ J_{k_0}, $ and that automatically gives that $ t_{k_0}\leq t^*. $ We conclude, that $ t'_{k_0-1}\leq t_{k_0}\leq t^*$ therefore $ t_{k_0}\in J'_{k_0} $. By the definition of $ t'_{k_0}\equiv\sup J'_{k_0}, $ it follows that, $ t_{k_0}\leq t'_{k_0}. $
		 		
		 		\item[(iv)] Let us suppose that, there exists $ k_0\in\N $ such that $ t_{k_0}=t^*$. The result will be proved through the Induction Principle. For $ k=k_0 $ we have, $ t_{k_0}=t^*. $ Then, $ t'_{k_0}=\sup J'_{k_0}  $ where $ J'_{k_0}=\sup\left\lbrace t\in[t'_{k_0-1},t^*]:\ \gamma(t)=\gamma(t^*)\right\rbrace  $. It holds that, $ t^*\in J'_{k_0} $ and $ t^* $ is an upper bound of $ J'_{k_0}. $ Therefore, $ t'_{k_0}=\sup J'_{k_0}=\max J'_{k_0}=t^*. $ Now, we suppose that for random $ k\geq k_0 $ that holds $ t_k=t^*=t'_k. $ We will prove that $ t_{k+1}=t^*=t'_{k+1.} $ Indeed since, $ t'_k=t^* $ it follows that $ \gamma(t'_k) =\gamma(t^*).$ Then from the definition of $ t_{k+1}, $ we obtain that $ t_{k+1}=t^*. $ Furthermore, $ t'_{k+1}=\sup J'_{k+1}=\sup\left\lbrace t\in[t'_k,t^*]:\ \gamma(t)=\gamma(t^*) \right\rbrace= t^*, $ since $ t^* $ is an upper bound of $ J'_{k+1} $ and also $ t^*\in J'_{k+1}. $ In conclution, from the Induction Principle, we get the desired result.\\
		 		 Next, let $ k_0\in\N  $ such that, $ t'_{k_0}=t^*. $ Then by the definition of $ t_{k_0+1} $ it follows that $ t_{k_0+1}=t^*. $ From the previous analysis, we receive that, $ \forall k\geq k_0+1, t_{k}=t^*=t'_k $.
		 		
		 		\item[(v)]
		 		\begin{itemize}
		 			\item[(a)]Let $ k^*\in\N. $ Initially, we observe that the desired inequality is true for $ k=1, $ since from Claim \ref{cl_0}(ii) we get $ t_x<t^*. $  Let us suppose the opposite of what we seek to prove i.e there exists $ k_0\in\{2,\dots,k^*\} $ such that $ t_{k_0-1}\geq t^* $. We already know from Claim \ref{cl_1}(iii) that $ t_{k_0-1}\leq t^* $. Therefore, $ t_{k_0-1}=t^* $ and $ \gamma(t_{k_0-1})=\gamma(t'_{k_0-1})=\gamma(t^*). $ As a result from the last, we receive that $ k_0-1\in J $ therefore $ k^*:=inf J\leq k_0-1 $, which is a contradiction, since $ k_0-1<k^*. $ With similar way we prove the desired inequality for the case where $ k^*=\infty. $ Moreover, we prove with analogous way the inequality $ t'_{k-1}<t^*,\ \forall k\in K. $
		 			
		 			\item[(b)]Let $ k^*\in\N.$  Initially, we observe that the desired inequality is true for $ k=1, $ since, $ \gamma(t_0)=\gamma(t_x)=x\notin M $ by assumption. At the same time, we know from Claim \ref{cl_0}(i) that $ \gamma(t^*)\in M$ and as a consequence $ \gamma(t_0)\neq\gamma(t^*). $ Let us suppose the opposite of what we seek to prove i.e there exists $ k_0\in\{2,\dots,k^*\} $ such that $ \gamma(t_{k_0-1})=\gamma( t^*) $. Clearly the last contradicts with the definition of $ k^*, $ since $ k_0-1<k^*. $  With similar way we prove the desired inequality for the case where $ k^*=\infty. $
		 			
		 			\item[(c)]  Let $ k^*\in\N $ and we select randomly $ k_0\in\{1,2,\dots,k^*\}.$ We will prove that, $ t'_{k_0-1}<t_{k_0}. $ For this, we distinguish the following cases for $ t_{k_0}. $
		 			\begin{itemize}
		 				\item $( t_{k_0}=t^*) $ Then from Claim \ref{cl_1} v(a) we obtain that $ t'_{k_0-1}<t^*=t_{k_0} $ since $ k_0-1<k_0\leq k^* $ therefore, $ t'_{k_0-1}< t_{k_0}.$
		 				\item $ (t_{k_0}<t^*) $ From the Claim \ref{cl_1} v(b) we obtain that $ \gamma(t'{k_0-1}) \neq \gamma(t^*).$ In consequence, $ t_{k_0}=\inf J_{k_0} $ where $ J_{k_0}:= \left\lbrace t\in[t'_{k_0-1},t^*]:\ \gamma(t)\in M\setminus \left(\cup_{i=0}^{k_0-1}\{\gamma(t'_i)\} \right)  \right\rbrace.  $ It is clear that $ t'_{k_0-1} $ is a lower bound of $ J_{k_0} $ and as a result by the definition of $ t_{k_0}=\inf J_{k_0}, $ we receive that $ t'_{k_0-1}\leq t_{k_0}. $ In specific, we will prove that $ t'_{k_0-1}<t_{k_0}. $ Let $ t'_{k_0-1}=t_{k_0} $ therefore, $ x_{k_0-1}=\gamma(t_{k_0-1})=\gamma(t'_{k_0-1})=\gamma(t_{k_0}) $. From the assumption of Theorem  \ref{main_theorem_1_b}, it holds that $ x_{k_0-1}\in M=iso(M). $ In consequence, there exists $ \delta_{k_0-1}>0 $ such that $ B_d\left(x_{k_0-1},\delta_{k_0-1} \right)\cap M\setminus\{x_{k_0-1}\}=\emptyset.  $ Due to the fact that, $ t_{k_0}=\inf J_{k_0} $ is a limit point of the set $ J_{k_0}, $ there exists a sequnece $ \left( \tilde{t}_l\right)_{l\in\N}\subset J_{k_0}  $ such that $ \tilde{t}_l\xrightarrow{l\to\infty}t_{k_0}. $ Furthermore, from the continuity of $ \gamma, $ it holds that $ \gamma(\tilde{t}_l)\xrightarrow{d}\gamma(t_{k_0}). $ Now, from the definition of $ J_{k_0}, $ since $ \left( \tilde{t}_l\right)\subset J_{k_0}  $ it follows that $ \forall l\in\N,\ \gamma(\tilde{t}_l)\in M\setminus\left(\cup_{i=0}^{k_0-1}\{\gamma(t'_i)\} \right) \subset M\setminus\{\gamma(t'_{k_0-1})\}=M\setminus\{x_{k_0-1}\}. $ From the convergence of $ \left(\gamma(\tilde{t}_l) \right)_{l\in\N},  $ there exists $ l_0\in\N,  $ such that, $ \forall l\geq l_0,\ \gamma(\tilde{t}_l)\in B_d\left(\gamma(t_{k_0}) ,\delta_{k_0-1}\right)=B_d\left( x_{k_0-1},\delta_{k_0-1}\right)  $. From the last, we conclude that, $B_d\left( x_{k_0-1},\delta_{k_0-1}\right)\cap M\setminus\{x_{k_0-1}\}\neq\emptyset  $ which contradicts with the fact that $ B_d\left(x_{k_0-1},\delta_{k_0-1} \right)\cap M\setminus\{x_{k_0-1}\}=\emptyset.  $ In a similar way we prove the desired inequality for the case where $ k^*=\infty. $
		 			\end{itemize}
		 			
		 			\item[(d)] Let $ k^*\in\N $ and we select randomly $ k_0\in\{1,2,\dots,k^*\}. $ Then from Claim \ref{cl_1} (iii), we obtain that $ t_{k_0-1}\leq t'_{k_0-1} $. But from Claim \ref{cl_1} v(c) we have also  $ t'_{k_0-1}<t_{k_0} $. Combining the last two, we receive the desired inequality. In a similar way we prove the desired inequality for the case where $ k^*=\infty. $
		 			
		 			\item [(e)] Let $ k^*\in\N $ and we select randomly $ k_0\in\{1,2,\dots,k^*\}. $ Then from Claim \ref{cl_1} v (c), we obtain that $ t'_{k_0-1}< t_{k_0} $. But from Claim \ref{cl_1} (iii) we have also  $ t_{k_0}\leq t'_{k_0} $. Combining the last two, we receive the desired inequality. In a similar way we prove the desired inequality for the case where $ k^*=\infty. $
		 		\end{itemize}
	 		  
	 		    \item[(vi)] Let $ k^*\in\N. $ We suppose the opposite of what we seek to prove i.e there are $ i_0,j_0\in K:=\{1,2,\dots,k^*\} $ such that $ x_{i_0}=\gamma(t_{i_0})=\gamma(t_{j_0}):=x_{j_0}. $ Without loss of generality, let $ i_0<j_0. $ If $ j_0=k^* $ then $ \gamma(t_{j_0})=\gamma(t_{k^*})=\gamma(t^*). $ But from the Claim \ref{cl_1}v(b) we conclude that $ \gamma(t_{i_0})\neq\gamma(t^*)=\gamma(t_{j_0}), $ since $ i_0<j_0=k^* $. The last result contracicts to our initial assumption $\gamma(t_{i_0})=\gamma(t_{j_0})  $. Now let $ j_0<k^*. $ By the assumption of theorem, we know that $ x_{i_0}\in iso(M). $ From the last, we obtain that there exists $ \delta_{i_0}>0 $ such that $ B_d\left(x_{i_0},\delta_{i_0} \right)\cap M\setminus \left\lbrace x_{i_0}\right\rbrace=\emptyset   $. Since $ j_0-1<k^*$, from the Claim \ref{cl_1}v(b) we receive $ \gamma(t_{j_0-1})\neq\gamma(t^*).  $ Therefore, $ t_{j_0}=\inf J_{j_0} $ where $ J_{j_0}:=\inf\left\lbrace t\in[t'_{j_0-1},t^*]:\ \gamma(t)\in M\setminus\left( \cup_{i=0}^{j_0-1}\{\gamma(t'_i)\}\right) \right\rbrace.   $ Since $ \inf J_{j_0} $ is a limit point of $ J_{j_0}, $ there exists a sequence  $ \left(\tilde{t}_l \right)_{l\in\N}\subset J_{j_0} $ such that $ \tilde{t}_l\xrightarrow{l\to\infty}\inf J_{j_0}\equiv t_{j_0} $. Moreover, from the continuity of $ \gamma $, we obtain that $ \gamma(\tilde{t}_l)\xrightarrow{l\to\infty}\gamma(t_{j_0}).  $ Now, by the definition of $ J_{j_0}, $ it follows that $\forall l\in\N,\ \gamma(\tilde{t}_l)\in M\setminus\left(\cup_{i=0}^{j_0-1}\{\gamma(t'_i)\} \right)\subset M\setminus\{\gamma(t'_{i_0})\}  $ where the last inclusion holds from the fact $ i_0<j_0\implies i_0\leq j_0-1. $ From the convergence of $ \left(\gamma\left( \tilde{t}_l\right)  \right)_{l\in\N} $ there exists $ l_0\in\N $ such that for all $ l\geq l_0, \gamma(\tilde{t}_l)\in B_d\left(\gamma(t_{j_0}),\delta_{i_0}\right) =B_d\left(\gamma(t_{i_0}),\delta_{i_0} \right) $. Consequently, $ \emptyset\neq B_d\left(\gamma(t_{i_0}),\delta_{i_0} \right)\cap M\setminus\{\gamma(t'_{i_0})\}=B_d\left(\gamma(t_{i_0}),\delta_{i_0} \right)\cap M\setminus\{\gamma(t_{i_0})\}\ $ where the last contradicts with fact that $ B_d\left(x_{i_0},\delta_{i_0} \right)\cap M\setminus \left\lbrace x_{i_0}\right\rbrace=\emptyset.$  In a similar way we prove the desired inequality for the case where $ k^*=\infty. $
	 		    
	 		    \item[(vii)] Let $ k^*\in\N. $ Let us suppose the opposite of what we seek to prove i.e. there exists $ k_0\in K:=\{1,2,\dots,k^*\} $ such that $ \gamma\left( \left( t'_{k_0-1},t_{k_0}\right) \right) \cap M\neq\emptyset $. Therefore, there exists $ z\in M $ such that $ z=\gamma(s) $ with $ t'_{k_0-1}<s<t_{k_0} $. From the Claim \ref{cl_1} v (b) it holds that $ \gamma(t'_{k_0-1})\neq\gamma(t^*) $ since $ k_0-1<k^*. $ As a result, $ t_{k_0}=\inf J_{k_0} $. We claim that $ s\in J_{k_0}. $ For the last it is sufficent to prove that, $ \gamma(s)\neq\gamma(t'_i),\ \forall i\in\{0,1,\dots,k_0-1\}. $ If $ i=0 $ then $ z:=\gamma(s)\neq \gamma(t'_0)=\gamma(t_x):=x $ since $ x\in U\setminus M $ and  $ z\in M. $ It remains to prove $ \gamma(s)\neq\gamma(t'_i),\ \forall i\in\{1,\dots,k_0-1\} $ with $ k_0\geq 2. $ Indeed, let us suppose that the last does not hold then there exists $ l\in\{1,2,\dots,k_0-1\} $ such that $ \gamma(s)=\gamma(t'_l). $ We have that $ t'_l:=\sup J'_{l},  $ where $ J'_l:=\left\lbrace t\in[t'_{l-1},t^*]:\ \gamma(t)=\gamma(t_l) \right\rbrace. $ Since, $ t'_{l-1}\leq^{\left(\text{Claim}\ \ref{cl_1} v(e) \right) } t'_{k_0-1}<s<t_{k_0}<^{\left( \text{Claim}\ref{cl_1} (iii)\right)}t^* $ we obtain that $ s\in J'_{l}. $ Therefore, $ t'_{l}\geq s. $ But the last contradicts with the fact that $ t'_l\leq^{\left(\text{Claim}\ \ref{cl_1} v(e) \right) } t'_{k_0-1}<s. $ Consequently, we receive that $ z=\gamma(s)\in M\setminus\left(\cup_{i=0}^{k_0-1}\left\lbrace \gamma(t'_i) \right\rbrace  \right)  $ and since $ s\in\left( t'_{k_0-1},t_{k_0}\right),  $ it follows that $ s\in J_{k_0} $. Therefore, $ t_{k_0}=\inf J_{k_0}\leq s $. The last leads to a contradiction since $ s<t_{k_0}. $ In a  similar way we prove the desired inequality for the case where $ k^*=\infty. $
 		  		 	
 	  		 	\end{itemize}		
		 \end{proof}
			 
    We define, $ x_k:=\gamma(t_k),\ k\in K $ and $ x^*:=\gamma(t^*) $. Let $ k\in K $. From the assumption, $ x_k\in int(U) $ and $ x_k\in iso(M) $. Therefore, there exist $ \epsilon_{x_k}>0,\tilde{\delta}_{x_k}>0 $ such that $ B_d\left( x_k,\epsilon_{x_k}\right)\subset U  $ and $  B_d\left( x_k,\tilde{\delta}_{x_k}\right) \cap M=\{x_k\} $.	
	For the construction of the path $ \tilde{\gamma}:[t_x,t_y]\to U\setminus M, $ connecting the points $ x, y, $ we distinguish the following cases for $ k^*. $\\
	\underline{\textbf{Case}}: $ k^*\in\N $\\
		 Equivalently, we have  $ K:=\{1,2,\dots,k^*\} $. In this case, we have $ x_{k^*}:=\gamma(t_{k^*})=\gamma(t^*)=x^* $.\\
		 We begin the construction of parts of the path $ \hat{\gamma}:[t_x,t_y]\to U\setminus M. $
		 
	\begin{itemize}
		\item[(i)] $ (k^*=1) $. Then $ K=\{1\} $ and $ \gamma(t_1)=\gamma(t'_1)=\gamma(t^*)\equiv x^* $. \\
		We define:
		\begin{gather}
				\tilde{t}^*:=\inf\left\lbrace t\in[t_x,t_1]:\ \gamma(t)=\gamma(t_1)\right\rbrace\equiv\inf\tilde{S}_{k^*}\nonumber.
		\end{gather}
		We notice that, $ \tilde{S}_{k^*}\neq\emptyset, $ since $ t_1\in \tilde{S}_{k^*}. $ Moreover, $ \tilde{S}_{k^*} $ is lower bounded by $ t_x $ thus $ \tilde{t}^* $ is well defined. Furthermore, by the continuity of $ \gamma, $ we obtain that $ \gamma(\tilde{t}^*)=\gamma(t_1). $ We claim also that $ t_1=\tilde{t}^* $. Indeed, from the Claim \ref{cl_1} (iii) we obtain $ t_x<t_1 $. In addition, from Claim \ref{cl_1} (vii), we obtain that $ \gamma\left( \left(t'_0,t_1 \right) \right)\cap M=\gamma\left( \left(t_x,t_1 \right) \right)\cap M=\emptyset  $. In consequence, $ \forall z\in (t_x,t_1),\ \gamma(z)\notin M $ and as a result $ \gamma(z)\neq x^*\in M. $ From the last, we receive that, $ \tilde{S}_{k^*}=\{t_1\} $ resulting that, $ t_1=\tilde{t}^*. $      
	Next, we define:
		\begin{gather}
			\delta_{x_{k^*}}:=\frac{1}{2}\min\left\lbrace\epsilon_{x_{k^*}},\tilde{\delta}_{x_{k^*}},d(x_{k^*},x),d(x_{k^*},y)\right\rbrace\nonumber.
		\end{gather} 
	 Clearly, we see that,
	 	\begin{gather}
				\delta_{x_{k^*}}\leq \delta'_{x_{k^*}}:=\frac{1}{2}\min\left\lbrace\epsilon_{x_{k^*}},d(x_{k^*},x) \right\rbrace\ \text{and}\
			\delta_{x_{k^*}}\leq\delta''_{x_{k^*}}:=\frac{1}{2}\min\left\lbrace \epsilon_{x_{k^*}},d(x_{k^*},y) \right\rbrace\nonumber
		\end{gather}
	where $ x_{k^*}\neq x $ and $  x_{k^*}\neq y, $ since $ x_{k^*}\in M $ and $ M\cap \{x,y\}=\emptyset. $ Furthermore, $ t_x<^{\left( \text{Claim}\ \ref{cl_1}(iii)\right) }t_1=\tilde{t}^* $ and $ t^*<t_y $
	Then, from the Proposition \ref{main_prop}, $ \forall\delta\in(0,\delta_{x_{k^*}}] $
	\begin{gather}
		\emptyset\neq K'_{x_{k^*}}:=\gamma\left([t_x,\tilde{t}^*]\right)  \cap\partial B_d\left( x_{k^*},\delta\right) \subset U\setminus M\ \text{and}\
		\emptyset \neq K''_{x_{k^*}}:=\gamma\left([ t^*,t_y]\right)  \cap\partial B_d\left( x_{k^*},\delta\right) \subset U\setminus M\nonumber.
	\end{gather}
      We fix such $ \delta_{k^*}\in (0,\delta_{x_{k^*}}) $. Therefore, 
      \begin{gather}
      	\exists z'_{k^*}\in[t_x,\tilde{t}^*]:\ \gamma(z'_{k^*})\in\partial B_d\left(x_{k^*},\delta_{k^*} \right),\ 
      	\exists z''_{k^*}\in[t^*,t_y]:\ \gamma(z''_{k^*})\in\partial B_d\left(x_{k^*},\delta_{k^*} \right)\nonumber.
      \end{gather}
      We define,
      \begin{gather}
      	t_{x'_{k^*}}:=\inf\left\lbrace t\in[t_x,\tilde{t}^*]:\ \gamma(t)\in\partial B_d(x_{k^*},\delta_{k^*}) \right\rbrace\equiv\inf S_{k^*},\nonumber\\
      	t_{x''_{k^*}}:=\sup\left\lbrace t\in[t^*,t_y]:\ \gamma(t)\in\partial  B_d(x_{k^*},\delta_{k^*})\right\rbrace\equiv\sup S'_{k^*}\nonumber\\
      	x'_{k^*}:=\gamma(t_{x'_{k^*}}),\ \text{and}\ x''_{k^*}:=\gamma(t_{x''_{k^*}})\nonumber. 
      \end{gather}
       Next, we observe that $ S_{k^*}\neq\emptyset, $ since $ z'_{k^*}\in S_{k^*} $ and $ S'_{k^*}\neq\emptyset, $ since $ z''_{k^*}\in S'_{k^*}. $ Moreover, $ S_{k^*} $ is lower bounded from $ t_x $ and $ S'_{k^*} $ is upper bounded by $ t_y $. As a result from the last, $ t_{x'_{k^*}}, t_{x''_{k^*}} $ are well defined. We mention that, by their definition of $ t_{x'_{k^*}}, t_{x''_{k^*}} $, we receive also, $ t_x\leq t_{x'_{k^*}} $ and $ t_{x''_{k^*}}\leq t_y $.  Furthermore, using the continuity of $ \gamma $ and the fact that $ \partial B_d(x_{k^*},\delta_{k^*}) $ is closed set, we conclude that $x'_{k^*}\equiv\gamma(t_{x'_{k^*}})\in \partial B_d(x_{k^*},\delta_{k^*} )\subset U\setminus M\ \text{and}\ x''_{k^*}\equiv\gamma(t_{x''_{k^*}})\in \partial B_d(x_{k^*},\delta_{k^*})\subset U\setminus M. ${\footnote{\text{We know that} $ \delta_{k^*}<\epsilon_{x_{k^*}}$ \text{and} $ \delta_{k^*}<\tilde{\delta}_{x_{k^*}} $. \text{Therefore}, $  \partial B_d(x_{k^*},\delta_{x_{k^*}})\subset\bar{B}_d(x_{k^*},\delta_{k^*})\subset B_d(x_{k^*},\epsilon_{x_{k^*}})\subset U $ and $ \partial B_d(x_{k^*},\delta_{k^*} )\subset\bar{B}_d(x_{k^*},\delta_{k^*})\subset B_d(x_{k^*},\tilde{\delta}_{x_{k^*}})$. Therefore, $ \partial B_d(x_{k^*},\delta_{k^*} )\cap M\subset B_d(x_{k^*},\tilde{\delta}_{x_{k^*}})\cap M=\{x_{k^*}\} $ but $ \partial B_d(x_{k^*},\delta_{k^*} )\cap\{x_{k^*}\}=\emptyset $. In consequence, $ \partial B_d(x_{k^*},\delta_{k^*} )\cap M=\emptyset $. We conclude then, $ \partial B_d(x_{k^*},\delta_{k^*} )\subset U\setminus M.  $}} From the last, it is extracted that $ t_x<t_{x'_{k^*}},\ t_{x''_{k^*}}<t_y $. Indeed, since $ x=\gamma(t_x)\in B_d(x,\frac{1}{2}d(x,x^*)) $ and $ x'_{k^*}=\gamma(t_{x'_{k^*}})\in B_d(x_{k^*},\delta_{k^*})\subset B_d(x_{k^*},\frac{1}{2}d(x,x_{k^*})) $ where, $B_d(x,\frac{1}{2}d(x,x_{k^*}))\cap B_d(x^*,\frac{1}{2}d(x,x_{k^*}))=\emptyset. $ Therefore, $ t_x\neq t_{x'_{k^*}} $ and as  result from above $ t_x\leq t_{x'_{k^*}} $, we receive that $ t_x< t_{x'_{k^*}} $. Similarly, we prove that $ t_{x''_{k^*}}<t_y. $  Using now the assumption of theorem, the boundary of ball $ B_d\left(x_{k^*},\delta_{k^*} \right)\subset U\setminus M  $ is path-connected. Therefore, there exists a continuous function $ \tilde{\gamma}_{x'_{k^*},x''_{k^*}}:[t_{x'_{k^*}},t_{x''_{k^*}}]\to\partial B_d\left(x_{k^*},\delta_{k^*} \right)  $ such that, $ \tilde{\gamma}_{x'_{k^*},x''_{k^*}}(t_{x'_{k^*}}) =x'_{k^*}$ and $\tilde{\gamma}_{x'_{k^*},x''_{k^*}}(t_{x''_{k^*}}) =x''_{k^*}.$ Notice that, from Claim \ref{cl_0}, $ \gamma\left( t^*,t_y\right)\cap M=\emptyset,  $ resulting that $ \gamma\left( \left(t_{x''_{k^*}},t_y \right)\right)  \cap M=\emptyset. $ Another important information, is the fact that, $ \gamma\left( \left(t_x,t_{x'_{k^*}} \right)\right)  \cap M=\emptyset. $ Indeed, this is obvious from the fact that, $ (t_x,t_{x'_{k^*}})\subset(t_x,\tilde{t}^*)=(t_x,t_1) $ and $ \gamma\left(\left( t_x,t_1\right)  \right)\cap M=\emptyset,  $ from the Claim \ref{cl_1} (vii).  \\
      We define the path $ \tilde{\gamma}:[t_x,t_y]\to U\setminus M, $ as follows:
      \begin{equation}
      	\tilde{\gamma}(t):=\begin{cases}
      		\gamma(t),\ & t\in[t_x,t_{x'_{k^*}}]\nonumber\\
      		\tilde{\gamma}_{x'_{k^*},x''_{k^*}}(t),\ & t\in[t_{x'_{k^*}},t_{x''_{k^*}}]\nonumber\\
      		\gamma(t),\ & t\in[t_{x''_{k^*}},t_y]\nonumber.
      	\end{cases}
      \end{equation}
      \\
       \item[(ii)] $ (k^*\geq 2) $, equivalently, $ \{1,2\}\subset K. $\\
		We define, $ \forall k\in K, $
		\begin{gather}
			r_k:=\min\left\lbrace d(x_k,x_i):\ i\in K\setminus\{k\}\right\rbrace\ \text{and}\
			\delta_{x_k}:=\frac{1}{2}\min\left\lbrace\epsilon_{x_k},\tilde{\delta}_{x_k},d(x_k,x),d(x_k,y),r_k \right\rbrace\nonumber.
		\end{gather}
		We make the following observations for the open balls $ B_d\left( x_k,\delta_{x_k}\right),\ k\in K $.
		\begin{itemize}
			\item[(i)] $ \forall k\in K,\ \forall\delta\in(0,\delta_{x_k}],\ \bar{B}_d\left( x_k,\delta\right)\subset U   $ 
			\item[(ii)] $ \forall k\in K,\ \forall\delta\in(0,\delta_{x_k}],\ \bar{B}_d\left( x_k,\delta\right)\cap M=\{x_k\}  $
			\item[(iii)]  $ \forall k\in K,\ \ \forall\delta\in(0,\delta_{x_k}],\ \bar{B}_d\left( x_k,\delta\right)\cap\left\lbrace x,y\right\rbrace  =\emptyset  $
			\item[(iv)]  $ \forall i,j\in K, i\neq j,\ \forall\delta\in(0,\delta_{x_k}],\forall\tilde{\delta}\in(0,\delta_{x_j}],\ B_d\left( x_i,\delta\right)\cap B_d(x_j,\tilde{\delta})=\emptyset  $
			\item[(vi)] $ \forall k\in K,\ \ \forall\delta\in(0,\delta_{x_k}],\ \partial B_d\left( x_k,\delta\right)\subset U\setminus M $.
		\end{itemize}

		 Clearly we see that, 
		\begin{gather}
			\delta_{x_1}\leq \delta'_{x_1}:=\frac{1}{2}\min\left\lbrace\epsilon_{x_1},d(x_1,x) \right\rbrace\ \text{and}\
			\delta_{x_1}\leq\delta''_{x_1}:=\frac{1}{2}\min\left\lbrace \epsilon_{x_1},d(x_1,x_2) \right\rbrace\nonumber.
		\end{gather}
		Then Proposition \ref{main_prop}, $ \forall\delta\in(0,\delta_{x_1}] $
		\begin{gather}
			\emptyset\neq K'_{x_1}:=\gamma\left([t_x,t_1]\right)  \cap\partial B_d\left( x_1,\delta\right) \subset U\setminus M\ \text{and}\
			\emptyset \neq K''_{x_1}:=\gamma\left([ t'_1,t_2]\right)  \cap\partial B_d\left( x_1,\delta\right) \subset U\setminus M\nonumber.
		\end{gather}
		We fix such $ \delta_1\in(0,\delta_{x_1}) $. Therefore, 
		\begin{gather}
			\exists z'_1\in[t_x,t_1]:\ x'_1:=\gamma(z'_1)\in\partial B_d\left(x_1,\delta_1 \right),\ 
			\exists z''_1\in[t'_1,t_2]:\ x''_1:=\gamma(z''_1)\in\partial B_d\left(x_1,\delta_1 \right)\nonumber.
		\end{gather}
		We define 
		\begin{gather}
			t_{x'_1}:=\inf\left\lbrace t\in[t_x,t_1]:\ \gamma(t)\in\partial B_d(x_1,\delta_1) \right\rbrace\equiv\inf S_1,\nonumber\\
			t_{x''_1}:=\sup\left\lbrace t\in[t'_1,t_2]:\ \gamma(t)\in\partial  B_d(x_1,\delta_1)\right\rbrace\equiv\sup S'_1\nonumber\\
			x'_1:=\gamma(t_{x'_1}),\ \text{and}\ x''_1:=\gamma(t_{x''_1})\nonumber.
		\end{gather}
		First, we observe that $ S_1\neq\emptyset, $ since $ z'_1\in S_1 $ and $ S'_1\neq\emptyset, $ since $ z''_1\in S'_1. $ Moreover, $ S_1 $ is lower bounded by $ t_x $ and $ S'_1 $ is upper bounded by $ t_2 $. As a result from the last, $ t_{x'_1}, t_{x''_1} $ are well defined. Furthermore, using the continuity of $ \gamma $ and the fact that $ \partial B_d(x_1,\delta_1) $ is closed set, we conclude that, $\gamma(t_{x'_1})\in \partial B_d(x_1,\delta_1 )\subset U\setminus M,\ \text{and}\ \gamma(t_{x''_1})\in \partial B_d(x_1,\delta_1)\subset U\setminus M. $  Now, we claim that $ t_{x'_1}<t_1\leq t'_1<t_{x''_1} $. Obviously, it holds $ t_1\leq t'_1 $ from Claimi \ref{cl_1}(iii). At first, we notice that $ t_1 $ is an upper bound of $ S_1 $ and $ t'_1 $ is a lower bound of $ \tilde{S}_1 $. Therefore, from the definition of $ t_{x'_1} $ and $ t_{x''_1} $, it follows that $ t_{x'_1}\leq t_1 $ and $ t'_1\leq t_{x''_1} $. Next we observe that,
		\begin{gather}
		x'_1\equiv\gamma(t_{x'_1})\in \partial B_d(x_1,\delta_1)\ \text{and}\ x''_1\equiv\gamma(t_{x''_1})\partial B_d(x_1,\delta_1)\nonumber\\ 
		x_1\equiv\gamma(t_1)=\gamma(t'_1)\notin\partial B_d(x_1,\delta_1).
		\end{gather}
		From the last, we receive that, $ t_{x'_1}<t_1  $ and $ t'_1<t_{x''_1}.  $
		 Using now the assumption of theorem, the boundary of ball $ B_d\left(x_1,\delta_1 \right)\subset U\setminus M  $ is path-connected. Therefore, there exists a continuous function $ \tilde{\gamma}_{x'_1,x''_1}:[t_{x'_1},t_{x''_1}]\to\partial B_d\left(x_1,\delta_1 \right)\subset U\setminus M  $ such that, $ \tilde{\gamma}_{x'_1,x''_1}(t_{x'_1}) =x'_1$ and $\tilde{\gamma}_{x'_1,x''_1}(t_{x''_1}) =x''_1.  $ Another information to mention, is the fact that, $ \gamma\left((t_x,t_{x'_1}) \right)\cap M=\emptyset.  $ Indeed, since $ (t_x,t_{x'_1})\subset(t_x,t_1) $ and from Claim \ref{cl_1} v(ii), we know that $ \gamma\left((t_x,t_1) \right)\cap M=\emptyset,  $ we conclude that, $ \gamma\left((t_x,t_{x'_1}) \right)\cap M=\emptyset. $
		\\
		The next step is the extraction of appropriate intermediate points $ (t_{x'_k})_{k=2,\dots,k^*},\ (t_{x''_k})_{k=2,\dots,k^*}\subset [t_x,t_y] $, using the Principle of Induction. In specific, we define inductively the following sequences of points $ (t_{x'_k})_{k=2,\dots,k^*} $ and  $ (t_{x''_k})_{k=2,\dots,k^*}  $ such that the corresponding sequences $ \left( \gamma(t_{x'_k})\right)_{k=2,\dots,k^*},\ \left(\gamma(t_{x''_k}) \right)_{k=2,\dots,k^*}\subset \partial B_d(x_k,\delta_k)   $ where $ \delta_k $ appropriate radious. We will prove the existence of the above sequences inductively.
		\begin{claim}\label{cl_2}
			Let, $ k^*\in\N, $ with $ k^*\geq 2. $ Then, for every $ k\in K\setminus\{1\}=\{2,\dots,k^*\}, $ there exists $ (t_{x'_i})_{i=2,\dots,k},\ (t_{x''_i})_{i=2,\dots,k}\subset [t_x,t_y] $ and $ \delta_i\in  (0,\delta_{x_i}),\ \forall i\in\{2,\dots,k\} $ such that:
			\begin{itemize}
				\item [(i)] If $ k\in K $ with $ k^*\geq 3 $ then, $\forall i\in\{2,3,\dots,k-1\},\ t_{x''_{i-1}}<t_{x'_i}<t_i\leq t'_i<t_{x''_i}<t_{i+1}$ and if $ k=k^* $ then $ t_{x''_{k^*}}<t_y $. 
				\item[(ii)]  $\forall i\in\{2,3,\dots,k\},\ x'_i:=\gamma(t_{x'_i}),\ x''_i:=\gamma(t_{x''_i})\in\partial B_d(x_i,\delta_i)\subset U\setminus M$
				\item[(iii)] $\forall i\in\{2,3,\dots,k\},\ \gamma\left( \left(t_{x''_{i-1}},t_{x'_i} \right) \right)\cap M=\emptyset,$ equivalently, $ \gamma\left( \left(t_{x''_{i-1}},t_{x'_i} \right) \right)\subset U\setminus M $
				\item[(iv)]If $ k=k^* $, then $ \gamma\left( (t_{x''_{k^*}},t_y)\right) \cap M=\emptyset $.
				\item[(v)] For all $ i\in\{1,2,\dots, k\} $, there exists a continuous function $ \tilde{\gamma}_{x'_k,x''_k}:[t_{x'_k},t_{x''_k}]\to\partial B_d\left(x_k,\delta_k \right)  $ such that, $ \tilde{\gamma}_{x'_k,x''_k}(t_{x'_k}) =x'_k$ and $\tilde{\gamma}_{x'_k,x''_k}(t_{x''_k}) =x''_k$.
			\end{itemize}
		\end{claim}
		\end{itemize}
		\begin{proof}
			We prove this by Induction. In specific, for $ k=2, $ we set $ \delta^*_{x_2}:=\frac{1}{2}\min\left\lbrace \delta_{x_2},d(x_2,x''_{1})\right\rbrace<\delta_{x_2} $. We claim that $ \delta^*_{x_2}>0 $, equivalently, $ d(x_2,x''_{1})>0 $, since we already know that, $ \delta_{x_2}>0. $  Indeed, $ x_2\in B_d\left( x_2,\delta_{x_2}\right) $ and $ x''_1\in\partial B_d\left( x_1,\delta_1\right)\subset B_d\left( x_1,\delta_{x_{1}}\right).   $ From property, $ (iv) $ for open balls, we receive that, $  B_d\left( x_2,\delta_{x_2}\right) \cap  B_d\left( x_1,\delta_{x_{1}}\right)=\emptyset. $ Therefore, $ x_2\neq x''_1. $ We observe that,
			\begin{gather}
				\delta^*_{x_2}\leq\frac{1}{2}\min\left\lbrace\epsilon_{x_2},d(x_2,x''_1) \right\rbrace\ \text{and}\ \delta^*_{x_2}\leq\begin{cases}
					\frac{1}{2}\min\left\lbrace\epsilon_{x_2},d(x_2,x_3) \right\rbrace,\ & k^*\geq 3 \nonumber\\
					\frac{1}{2}\min\left\lbrace\epsilon_{x_2},d(x_2,y) \right\rbrace,\ & k^*=2 \nonumber.
				\end{cases}
			\end{gather}
			Now, from Proposition \ref{main_prop}, $\forall \delta\in(0,\delta^*_{x_2}]$
			\begin{gather}
				\emptyset\neq K'_{x_2}:=\gamma\left([t_{x''_{1}},t_2]\right) \cap\partial B_d\left( x_2,\delta\right) \subset U\setminus M\ \text{and}\\
				\emptyset \neq K''_{x_2}:=\begin{cases}
					\gamma\left([t'_2,t_{3}]\right) \cap\partial B_d\left( x_2,\delta\right) \subset U\setminus M,\ & k^*\geq 3\nonumber\\
					\gamma\left([t^*,t_y]\right) \cap\partial B_d\left( x_2,\delta\right) \subset U\setminus M,\ & k^*=2 \nonumber.
				\end{cases}
			\end{gather}
			We fix such $ \delta_2\in(0,\delta^*_{x_2}). $ As a result from the last,
			\begin{gather}
				\exists z'_2\in[t_{x''_{1}},t_2],\ \gamma(z'_2)\in\partial B_d\left(x_2,\delta_2 \right),\ 
				\exists z''_2\in\begin{cases}
					[t'_2,t_{3}],\ & k^*\geq 3\nonumber\\
					[t^*,t_y],\ & k^*=2\nonumber
				\end{cases},\ \gamma(z''_2)\in\partial B_d\left(x_2,\delta_2 \right)\nonumber.
			\end{gather} 
			Next, we define
			\begin{gather}
			\begin{align}
				t_{x'_2}&:=\inf\left\lbrace t\in[t_{x''_{1}},t_2]:\ \gamma(t)\in\partial B_d(x_2,\delta_2)\right\rbrace\equiv\inf S_2\nonumber\\
				t_{x''_2}&:=\begin{cases}
					\sup\left\lbrace t\in[t'_2,t_{3}]:\ \gamma(t)\in\partial  B_d(x_2,\delta_2)\right\rbrace\equiv\sup S'^{(k^*\geq 3)}_2,\ & k^*\geq 3\nonumber\\
					\sup\left\lbrace t\in[t^*,t_y]:\ \gamma(t)\in\partial  B_d(x_2,\delta_2)\right\rbrace\equiv\sup S'^{(k^*=2)}_2,\ & k^*=2\nonumber.
							\end{cases}
			\end{align}
		\end{gather}
			\\
			We claim that $t_{x'_2},t_{x''_2}  $ are well defined. Indeed, $ S_2\neq\emptyset, $ since $ z'_2\in S_2 $ and $ S'_2\neq\emptyset, $ since $ z''_2\in S'^{(k^*\geq 3)}_2,\\ \left( \ z''_2\in S'^{(k^*=2)}_2\right)  $. Moreover, $ S_2 $ is lower bounded by $ t_{x''_1} $ and $ S'^{(k^*\geq 3)}_2 $ is upper bounded by $ t_3.  $ Also $ S'^{(k^*=2)}_2  $ is upper bounded by $ t_y $. Therefore, $t_{x'_2},t_{x''_2}  $ are well defined.  Furtheromore, using the continuity of $ \gamma $ and the fact that $ \partial B_d(x_2,\delta_2) $ is closed set, we conclude that, $x'_2\equiv\gamma(t_{x'_2})\in \partial B_d(x_2,\delta_2),\ \text{and}\ x''_2\equiv\gamma(t_{x''_2})\in \partial B_d(x_2,\delta_2). $  As we did in the proof for the case where $ k=1 $, we obtain again that $ t_{x'_2}<t_2\leq t'_2<t_{x''_2}. $
			Now, we proceed to the proof of $ t_{x''_1}<t_{x'_2} $. First, we notice that $ t_{x''_1} $ is a lower bound of $ S_2 $ and as a result from the definition of infimum, we obtain that, $ t_{x''_1}\leq t_{x'_2} $. We claim, that $ t_{x''_1}< t_{x'_2} $. For this, it is sufficent to prove that $ t_{x''_1}\neq t_{x'_2} $, since $ t_{x''_1} $ is a lower bound of $ S_2. $ At first, we see that, the following holds:
			\begin{gather}
				x''_1\in\partial B_d(x_1,\delta_1)\subset\bar{B}_d(x_1,\delta_1)\subset B_d(x_1,\delta_{x_1}),\ \text{since}\ \delta_1\leq\delta_{x_1}\nonumber\\
				x'_2\in\partial \bar{B}_d(x_2,\delta_2)\subset B_d(x_2,\delta_{x_2}),\ \text{since}\ \delta_2\leq\delta^*_{x_2}<\delta_{x_2}\nonumber.
			\end{gather}
			Using now property, $ (iv), $ we obtain that, $  B_d(x_1,\delta_{x_1})\cap B_d(x_2,\delta_{x_2})=\emptyset.  $ Therefore, $ x''_1\neq x'_2, $ resulting that, $ t_{x''_1}\neq t_{x'_2}, $ which leads to $ t_{x''_1}< t_{x'_2} $.	Moreover, we claim that $ t_{x''_2}<t_3 $ in case where $ k^*\geq 3. $ Since, $ t_3 $ is an upper bound of $ S'^{(k^*\geq 3)}_2, $ it follows that $ t_{x''_2}\leq t_3 $. It remains to prove that, $ t_{x''_2}\neq t_3 $. Indeed, we observe that:
			\begin{gather}
				 x''_2\in \partial B_d(x_2,\delta_2) \subset\bar{B}_d(x_2,\delta_2)\subset B_d(x_2,\delta_{x_2})\ \text{and}\ x_3=\gamma(t_3)\in B_d(x_3,\delta_{x_3}).\nonumber
			\end{gather}
			Using now property $ (iv) $, we obtain that $ B_d(x_2,\delta_{x_2})\cap B_d(x_3\delta_{x_3}). $ Therefore, $ x''_2\neq x_3, $ resulting that $ t_{x''_2}\neq t_3. $ If now, $ k^*=2, $  we prove similarly that $ t_{x''_2}<t_y. $ In addition, $ \gamma\left( (t_{x''_1},t_{x'_2})\right)\cap M=\emptyset.  $ Indeed, this is true from the fact that $ \gamma\left((t'_1,t_2) \right)\cap M=\emptyset  $, (see Claim \ref{cl_1} (vii)) and $ (t_{x''_1},t_{x'_2})\subset(t_{x''_1},t_2). $ Also in case, where $ k^*=2 $, then $ \gamma\left( (t_{x''_2},t_y)\right)\cap M=\emptyset,  $ since $ (t_{x''_2},t_y)\subset (t^*,t_y)  $ and from Claim \ref{cl_0} (iii), $ \gamma((t^*,t_y))\cap M=\emptyset. $ From the assumption of theorem, we know that the boundary of the ball $ B_d\left(x_{2},\delta_{2} \right)\subset U\setminus M  $ is path-connected. Therefore, there exists a continuous function $ \tilde{\gamma}_{x'_{2},x''_{2}}:[t_{x'_{2}},t_{x''_{2}}]\to\partial B_d\left(x_{2},\delta_{2} \right)  $ such that $ \tilde{\gamma}_{x'_{2},x''_{2}}(t_{x'_{2}}) =x'_{2}$ and $\tilde{\gamma}_{x'_{2},x''_{2}}(t_{x''_{2}}) =x''_{2}.$
			\\
			Next, let us suppose that for a random $ k'\in\{2,\dots,k^*-1\} $ with $ k^*\geq 3 $, we generated the following sequences, $ (t_{x'_i})_{i=2,\dots,k'},\  (t_{x''_i})_{i=2,\dots,k'}\subset [t_x,t_y] $ and also $ \delta\in(0,\delta_{x_i}),\ \forall i_{i=2,\dots,k'}  $ which satisfy properties $ (i)-(v) $ of the Claim. (Inductive Step). We set, $ x'_i:=\gamma(t_{x'_i}),x''_i:=\gamma(t_{x''_i}),\ \forall i\in\{2,3,\dots,k'\} $. The target now, is to prove the statement in the case where $ k=k'+1 $. In order to prove the last, it is sufficcent from the Inductive step, to prove the existence of terms $ t_{x'_{k'+1}},\ t_{x''_{k'+1}}\subset I $ and $ \delta_{k'+1}\subset (0,\delta_{x_{k'+1}}) $ satisfying the following:
			\begin{itemize}
			\item If,	$ k'+1<k^* $,then $ t_{x''_{k'}}<t_{x'_{k'+1}}<t_{k'+1}\leq t'_{k'+1}<t_{x''_{k'+1}} $ and if $ k'+1=k^* $,then $ t_{x''_{k'+1}}<t_y $ 
			\item $ \gamma(t_{x'_{k'+1}}),\ \gamma(t_{x''_{k'+1}})\in\partial B_d(x_{k'+1},\delta_{k'+1}) $
			\item $ \gamma\left((t_{x''_{k'}},t_{x'_{k'+1}})\right) \cap M=\emptyset  $
			\item If $ k'+1=k^* $, then $ \gamma\left( (t_{x''_{k'+1}},t_y)\right)\cap M=\emptyset $
			\item There exists a continuous path $ \tilde{\gamma}_{x'_{k'+1},x''_{k'+1}}:[t_{x'_{k'+1},t_{x''_{k'+1}}}]\to\partial B_d(x_{k'+1},\delta_{k'+1}), $ with \\ $ \tilde{\gamma}_{x'_{k'+1},x''_{k'+1}}(t_{x'_{k'+1}}) =x'_{k'+1} $ and $ \tilde{\gamma}_{x'_{k'+1},x''_{k'+1}}(t_{x''_{k'+1}})=x''_{k'+1} $.
			\end{itemize}
			We set $ \delta^*_{x_{k'+1}}:=\frac{1}{2}\min\left\lbrace\delta_{x_{k'+1}},d(x_{k'+1},x''_{k'}) \right\rbrace  $. We claim that $ \delta^*_{x_{k'+1}}>0, $ equivalently, $ d(x_{k'+1},x''_{k'})>0, $ since we already know that $ \delta_{x_{k'+1}}>0. $ lndeed, $ x_{k'+1}\in B_d(x_{k'+1},\delta_{x_{k'+1}}) $ and $ x''_{k'}\in\partial B_d(x_{k'},\delta_{k'})\subset B_d(x_{k'},\delta_{x_{k'}}). $ From property, (iv) for open balls, we receive that $ B_d(x_{k'+1},\delta_{x_{k'+1}})\cap B_d(x_{k'},\delta_{x_{k'}})=\emptyset.  $ Therefore, $ x_{k'+1}\neq x''_{k'}. $ Moreover, 
			\begin{gather}
				\delta^*_{x_{k'+1}}\leq\frac{1}{2}\min\left\lbrace \epsilon_{x_{k'}},d(x_{k'+1},x''_{k'})\right\rbrace\ \text{and}\ \delta^*_{x_{k'+1}}\leq\begin{cases}
					\frac{1}{2}\min\left\lbrace \epsilon_{x_{k'}},d(x_{k'+1},x_{k'+2})\right\rbrace,\ & k'+1<k^*\nonumber\\
					\frac{1}{2}\min\left\lbrace \epsilon_{x_{k'}},d(x_{k'+1},y)\right\rbrace,\ & k'+1=k^*\nonumber.
				\end{cases}
			\end{gather}
			Now, from Proposition \ref{main_prop}, $ \forall\delta\in(0,\delta^*_{x_{k'+1}}] $
			\begin{gather}
				\emptyset\neq K'_{x_{k'+1}}:=\gamma\left( [t_{x''_{k'}},t_{k'+1}] \right)\cap\partial B_d(x_{k'+1},\delta)\subset U\setminus M\ \text{and}\nonumber\\	\emptyset\neq K''_{x_{k'+1}}:=\begin{cases}
					\gamma\left( [t'_{k'+1},t_{k'+2}] \right)\cap\partial B_d(x_{k'+1},\delta)\subset U\setminus M,\ & k'+1<k^*\nonumber\\
					\gamma\left( [t^*,t_y] \right)\cap\partial B_d(x_{k'+1},\delta)\subset U\setminus M,\ & k'+1=k^*\nonumber.
				\end{cases}
			\end{gather}
			We fix such $ \delta_{k'+1}\in(0,\delta^*_{x_{k'+1}} ).  $ As a result from the last, 
			\begin{gather}
				\exists z'_{k'+1}\in[t_{x''_{k'}},t_{k'+1}],\ \gamma(z'_{k'+1})\in\partial B_d(x_{k'+1},\delta_{k'+1}),\nonumber\\ \exists z''_{k'+1}\in\begin{cases}
					[t'_{k'+1},t_{k'+2}],\ & k'+1<k^*\nonumber\\
					[t^*,t_y],\ &k'+1=k^*
				\end{cases},\ \gamma(z''_{k'+1})\in\partial B_d(x_{k'+1},\delta_{k'+1})\nonumber.
			\end{gather}
			Next, we define
			\begin{align}
				t_{x'_{k'+1}}&:=\inf\left\lbrace t\in[t_{x''_{k'}},t_{k'+1}]:\ \gamma(t)\in\partial B_d(x_{k'+1},\delta_{k'+1}) \right\rbrace\equiv\inf S_{k'+1}\nonumber\\
				t_{x''_{k'+1}}&:=\begin{cases}
					\sup\left\lbrace t\in[t'_{k'+1},t_{k'+2}]:\ \gamma(t)\in\partial B_d(x_{k'+1},\delta_{k'+1}) \right\rbrace\equiv\sup S'^{(k'+1<k^*)}_{k'+1},\ & k'+1<k^*\nonumber\\
					\sup\left\lbrace t\in[t^*,t_y]:\ \gamma(t)\in\partial B_d(x_{k'+1},\delta_{k'+1}) \right\rbrace\equiv\sup S'^{(k'+1=k^*)}_{k'+1},\ & k'+1=k^*\nonumber.
				\end{cases}
			\end{align}
		
		We set, $ x'_{k'+1}:=\gamma(t_{x'_{k'+1}}) $ and $ x''_{k'+1}:=\gamma(t_{x''_{k'+1}}). $
		We claim that $t_{x'_{k'+1}},t_{x''_{k'+1}}  $ are well defined. Indeed, $ S_{k'+1}\neq\emptyset, $ since $ z'_{k'+1}\in S_{k'+1} $ and $ S'^{(k'+1<k^*)}_{k'+1}\neq\emptyset, $ since $ z''_{k'+1}\in S'^{(k'+1<k^*)}_{k'+1},\\ \left(  z''_{k'+1}\in S'^{(k'+1=k^*)}_2\right)  $. Moreover, $ S_{k'+1} $ is lower bounded by $ t_{x''_{k'}} $ and $ S'^{(k'+1<k^*)}_{k'+1} $ is upper bounded by $ t_{k'+2}.  $ Also $ S'^{(k'+1=k^*)}_{k'+1}  $ is upper bounded by $ t_y $. Therefore, $t_{x'_{k'+1}},t_{x''_{k'+1}}  $ are well defined.  Furthermore, using the continuity of $ \gamma $ and the fact that $ \partial B_d(x_{k'+1},\delta_{k'+1}) $ is closed set, we conclude that, $x'_{k'+1}\equiv\gamma(t_{x'_{k'+1}})\in \partial B_d(x_{k'+1},\delta_{k'+1}),\ \text{and}\ x''_{k'+1}\equiv\gamma(t_{x''_{k'+1}})\in \partial B_d(x_{k'+1},\delta_{k'+1}). $  As we did in the proof for the case where $ k=1 $, we obtain again that $ t_{x'_{k'+1}}<t_{k'+1}\leq t'_{k'+1}<t_{x''_{k'+1}}. $
		Now, we proceed to the proof of $ t_{x''_{k'}}<t_{x'_{k'+1}} $. First, we notice that $ t_{x''_{k'}} $ is a lower bound of $ S_{k'+1} $ and as a result from the definition of infimum, we obtain that, $ t_{x''_{k'}}\leq t_{x'_{k'+1}} $. We claim, that $ t_{x''_{k'}}< t_{x'_{k'+1}} $. For this, it is sufficent to prove that $ t_{x''_{k'}}\neq t_{x'_{k'+1}} $, since $ t_{x''_{k'}} $ is a lower bound of $ S_{k'+1}. $ At first, we see that, the following holds:
		\begin{gather}
			x''_{k'}\in\partial B_d(x_{k'},\delta_{k'})\subset\bar{B}_d(x_{k'},\delta_{k'})\subset B_d(x_{k'},\delta_{x_{k'}}),\ \text{since}\ \delta_{k'}<\delta_{x_{k'}}\nonumber\\
			x'_{k'+1}\in\partial \bar{B}_d(x_{k'+1},\delta_{k'+1})\subset B_d(x_{k'+1},\delta_{x_{k'+1}}),\ \text{since}\ \delta_{k'+1}\leq\delta^*_{x_{k'+1}}<\delta_{x_{k'+1}}\nonumber.
		\end{gather}
		Using now property, $ (iv), $ we obtain that, $  B_d(x_{k'},\delta_{x_{k'}})\cap B_d(x_{k'+1},\delta_{x_{k'+1}})=\emptyset.  $ Therefore, $ x''_{k'}\neq x'_{k'+1}, $ resulting that, $ t_{x''_{k'}}\neq t_{x'_{k'+1}}, $ which leads to $ t_{x''_{k'}}< t_{x'_{k'+1}} $.	Moreover, we claim that $ t_{x''_{k'+1}}<t_{k'+2} $ in case where $ k^*\geq 3. $ Since, $ t_{k'+2} $ is an upper bound of $ S'^{(k'+1<k^*)}_{k'+1}, $ it follows that $ t_{x''_{k'+1}}\leq t_{k'+2} $. It remains to prove that, $ t_{x''_{k'+1}}\neq t_{k'+2} $. Indeed, we observe that:
		\begin{gather}
			x''_{k'+1}\in \partial B_d(x_{k'+1},\delta_{k'+1}) \subset\bar{B}_d(x_{k'+1},\delta_{k'+1})\subset B_d(x_{k'+1},\delta_{x_{k'+1}})\ \text{and}\ x_{k'+2}=\gamma(t_{k'+2})\in B_d(x_{k'+2},\delta_{x_{k'+2}})\nonumber.
		\end{gather}
		Using now property $ (iv) $, we obtain that $ B_d(x_{k'+1},\delta_{x_{k'+1}})\cap B_d(x_{k'+2}\delta_{x_{k'+2}}). $ Therefore, $ x''_{k'+1}\neq x_{k'+2}, $ resulting that $ t_{x''_{k'+1}}\neq t_{k'+2}. $ If now, $ k^*=k'+1, $  we prove similarly that $ t_{x''_{k'+1}}<t_y. $ In addition, $ \gamma\left( (t_{x''_{k'}},t_{x'_{k'+1}})\right)\cap M=\emptyset.  $ Indeed, this is true from the fact that $ \gamma\left((t'_{k'},t_{k'+1}) \right)\cap M=\emptyset  $, (see Claim \ref{cl_1} (vii)) and $ (t_{x''_{k'}},t_{x'_{k'+1}})\subset(t_{x''_{k'}},t_{k'+1}). $ Also in case, where $ k^*=k'+1 $ then $ \gamma\left( (t_{x''_{k'+1}},t_y)\right)\cap M=\emptyset,  $ since $ (t_{x''_{k'+1}},t_y)\subset (t^*,t_y)  $ and from Claim \ref{cl_0} (iii), $ \gamma((t^*,t_y))\cap M=\emptyset. $ From the assumption of theorem, we know that the boundary of the ball $ B_d\left(x_{k'+1},\delta_{k'+1} \right)\subset U\setminus M  $ is path-connected. Therefore, there exists a continuous function $ \tilde{\gamma}_{x'_{k'+1},x''_{k'+1}}:[t_{x'_{k'+1}},t_{x''_{k'+1}}]\to\partial B_d\left(x_{k'+1},\delta_{k'+1} \right)  $ such that, $ \tilde{\gamma}_{x'_{k'+1},x''_{k'+1}}(t_{x'_{k'+1}}) =x'_{k'+1}$ and $\tilde{\gamma}_{x'_{k'+1},x''_{k'+1}}(t_{x''_{k'+1}}) =x''_{k'+1}.$ By the use of Induction Principle, we conclude the existence of a finite sequence of points $ (t_{x'_i})_{i=2,\dots, k^*},  (t_{x''_i})_{i=2,\dots, k^*} \in I $ and $ \delta_i\in(0,\delta_{x_i}),\ \forall i=2,\dots,k^* $ which satisfy properties $ (i)-(v) $ of the Claim.
    \end{proof}
	Finally, we define for the case where $ k^*\geq 2 $, the path $ \hat{\gamma}:[t_x,t_y]\to U\setminus M $  as follows:
	\begin{equation}
	\hat{\gamma}(t):=\begin{cases}
		\gamma(t),\ & t\in[t_x,t_{x'_1}]\nonumber\\
		\tilde{\gamma}_{x'_1,x''_1}(t),\ & t\in[t_{x'_1},t_{x''_1}]\nonumber\\
		\gamma(t),\ & t\in[t_{x''_1},t_{x'_2}]\nonumber\\
		\vdots\     & \vdots\nonumber\\
		\gamma(t),\ & t\in[t_{x''_{k-1}},t_{x'_{k}}]\nonumber\\
		\tilde{\gamma}_{x'_k,x''_k}(t),\ & t\in[t_{x'_{k}},t_{x''_k}]\nonumber\\
		\gamma(t),\ & t\in[t_{x''_k},t_{x'_{k+1}}]\nonumber\\
		\vdots\ & \vdots\nonumber\\
		\gamma(t),\ & t\in[t_{x''_{k^*-1}},t_{x'_{k^*}}]\nonumber\\
		\tilde{\gamma}_{x'_{k^*},x''_{k^*}}(t),\ & t\in[t_{x'_{k^*}},t_{x''_{k^*}}]\nonumber\\
		\gamma(t),\ & t\in[t_{x''_{k^*}},t_y].
	\end{cases}
	\end{equation}
\\
\\
	\underline{\textbf{Case}}: $ k^*=\infty $\\
	Equivalently we have  $ K:=\N $. A direct concequence,for that case, is $ x_k\neq x^*,\ \forall k\in\N. $ In this setting we define again with a similar way, as we did before, appropriate sequence of points $ (t_{x'_k})_{k=1,2,\dots k_0},(t_{x''_k})_{k=1,2,\dots k_0}\subset I $ where $ k_0\in\N $ satisfying suitable properties. First, we define:
	   \begin{gather}
	   	\forall k\in\N\setminus\{1\},\ r_k:=\min\left\lbrace d(x_k,x_i):\ i=1,2\dots,k-1\right\rbrace
	   \end{gather}
     and
	\begin{gather}
		\delta_{x_k}:=\begin{cases}
			\frac{1}{2}\min\left\lbrace\epsilon_{x_k},\tilde{\delta}_{x_k},d(x_k,x),d(x_k,x^*),d(x_k,y) \right\rbrace,\ & k= 1\nonumber\\
			\frac{1}{2}\min\left\lbrace\epsilon_{x_k},\tilde{\delta}_{x_k},d(x_k,x),d(x_k,x^*),d(x_k,y),r_k \right\rbrace,\ & k\geq 2\nonumber.
		\end{cases}
	\end{gather}
  \\
  We proceed now to the main claim that describes the construction of the above sequences.
  
  \begin{claim}\label{cl_3}
  	For every $ k\in\N, $ there exists $ (t_{x'_i})i=1,2,\dots, k,\ (t_{x''_i})i=1,2,\dots, k\subset I $ and $ \delta_i \in(0,\delta_{x_i}),\ \forall i\in\{1,2,\dots ,k\}$ such that:
  	\begin{itemize}
  		\item[(i)] If $ k=1 $ it holds that: $ t_x<t_{x'_1}<t_1\leq t'_1<t_{x''_1}<t_2 $ and if $ k\geq 3 $ it holds that, $ t_{x''_{i-1}}<t_{x'_i}<t_i\leq t'_i<t_{x''_i}<t_{i+1},\ \forall i\in\{2,3,\dots,k\} $
  		\item[(ii)] $ \forall i\in\{1,2,\dots, k\},\  x'_i:=\gamma(t_{x'_i}),\ x''_i:=\gamma(t_{x''_i})\in\partial B_d(x_i,\delta_i)\subset U\setminus M$
  		\item[(iii)] If $ k=1 $ then $ \gamma\left( (t_x,t_{x'_1})\right) \cap M=\emptyset $ and if $ k\geq 2 $ then $ \forall i\in\{2,3,\dots,k\},\ \gamma\left((t_{x''_{i-1}},t_{x'_i}) \right)\cap M=\emptyset  $
  		\item[(iv)] For all $ i\in\{1,2,\dots, k\} $ there exists a continuous function $ \tilde{\gamma}_{x'_k,x''_k}:[t_{x'_k},t_{x''_k}]\to\partial B_d(x_k,\delta_k) $ such that $ \tilde{\gamma}_{x'_k,x''_k}(t_{x'_k})=x'_k $ and $ \tilde{\gamma}_{x'_k,x''_k}(t_{x''_k})=x''_k $.
  	\end{itemize}
  \end{claim}
  We omit the proof of the above Claim, since is very similar with the proof of the Claim \ref{cl_2}.
\\
We remind that in a general metric space holds\footnote{Let $ x_0\in X $, $ (x_j)_{j\in\N}\subset A, $ such that $ x_j\xrightarrow{d} x_0$. Therefore, $ x_0\in\bar{A}. $ Since $ A $ is closed, it follows that $ \bar{A}=A $ and as a result $ x_0\in A=iso(A). $ From the last, there exists $ \tilde{\delta}_{x_0}>0,\ B_d(x_0,\tilde{\delta}_{x_0})\cap A=\{x_0\} $. From the convergence of $ (x_j)_{j\in\N} $ to $ x_0, $ there exists $ j_0\in\N $ such that $ \forall j\geq j_0,\ x_j\in B_d(x_0,\tilde{\delta}_{x_0})\cap A=\{x_0\} $. I.e $ x_j=x_0,\ \forall j\geq j_0 $ where the last is a contradiction since $ x_i\neq x_l,\ \forall i,l\in\N, i\neq l . $ } the following:\\
\textit{Let $ (X,d) $ be a metric space, $ A\subset X, $ closed with $ iso(A)=A. $ If $ (x_j)_{j\in\N}\subset A $ with $ x_i\neq x_l,\ \forall i,l\in\N, i\neq l $, then for all $ x_0\in X,\ x_j\nrightarrow x_0 $}.\\
As a consequence from the above result, we conclude that $ A:=\{x_k:\ k\in\N\} $ is closed in $(X,d).  $ Moreover, since $ A\subset\gamma\left([t_x,t_y] \right),  $ where $ \gamma\left([t_x,t_y] \right)  $ is compact as a continuous image of a compact set, we conclude that $ A $ is also compact as a closed subset of a compact set. Furthermore, since $ \{x^*\} $ is a compact set, (as a finite), from a known result in metric spaces, it follows that 
\begin{gather}
	\exists k_0\in\N,\ dist(x^*, A):=\inf\left\lbrace d(x^*,s):\ s\in A\right\rbrace =d(x^*,x_{k_0}). 
	\end{gather}
Due to the fact that the sequence  $ (x_k)_{k\in\N}  $ consists of discrete points (i.e $ \forall i,l\in\N, i\neq l,\ x_i\neq x_l $) combined with $ (U-M) $ property of the set $ U ,$ we conclude that there exists a continuous path $ \gamma^*_{x_{k_0},x^*}:[t'_{k_0},t^*]\to U $ such that, $ x_{k_0}=\gamma^*_{x_{k_0},x^*}(t'_{k_0}) $ and $ x^*=\gamma^*_{x_{k_0},x^*}(t^*) $, with $ \gamma^*_{x_{k_0},x^*}\left(( t_{x'_{k_0}},t^*)\right)\cap M=\emptyset. $
 	
By the selection of $ \delta_{k_0} $ we obtain that $ \delta_{k_0}<\frac{1}{2}\min\{\epsilon_{x_{k_0}},d(x_{k_0},x_{k^*})\}$. Then from Proposition \ref{main_prop} it holds that 
\begin{gather}
	\emptyset\neq\tilde{K}''_{x_{k_0}}:=\gamma^*_{x_{k_0},x^*}([t'_{k_0},t^*])\cap \partial B_d(x_{k_0},\delta_{k_0}).
\end{gather}
	Therefore, there exists, $ \tilde{z}''_{k_0}\in[t'_{k_0},t^*],\ \gamma^*_{x_{k_0},x^*}(\tilde{z}''_{k_0})\in \partial B_d(x_{k_0},\delta_{k_0}) $.
	We define,
	\begin{gather}
		\tilde{t}_{x_{k_0}}:=\sup\left\lbrace t\in[t'_{k_0},t^*]:\ \gamma^*_{x_{k_0},x^*}(t)\in\partial B_d(x_{k_0},\delta_{k_0}) \right\rbrace\nonumber\\
		\tilde{x}''_{k_0}:=\gamma^*_{x_{k_0},x^*}(\tilde{t}_{x_{k_0}})\nonumber.
	\end{gather}
As we can easilly see, $ \tilde{t}_{x_{k_0}} $ is well defined and moreover from the continuity of $ \gamma^*_{x_{k_0},x^*} $ and the fact that $ \partial B_d(x_{k_0},\delta_{k_0})  $ is closed, we obtain that $ \tilde{x}''_{k_0}\in \partial B_d(x_{k_0},\delta_{k_0})\subset U\setminus M. $ Using now the information of the theorem, we obtain that the boundary of the ball $  B_d(x_{k_0},\delta_{k_0}) $ is path connected. Therefore, there exists a continuous function $ \tilde{\gamma}_{x'_{k_0},\tilde{x}''_{k_0}}:[t_{x'_{k_0}},\tilde{t}_{x_{k_0}}]\to\partial B_d(x_{k_0},\delta_{k_0})\subset U\setminus M, $ with $ x'_{k_0}= \tilde{\gamma}_{x'_{k_0},\tilde{x}''_{k_0}}(t_{x'_{k_0}})$ and $ \tilde{x}''_{k_0}=\tilde{\gamma}_{x'_{k_0},\tilde{x}''_{k_0}}(\tilde{t}_{x_{k_0}}). $\\
Due to the fact that, $ x^*\in M $, by the assumptions of theorem, there exist $ \epsilon_{x^*}>0,\tilde{\delta}_{x^*}>0 $ such that $ B_d(x^*,\epsilon_{x^*})\subset U $ and $ B_d(x^*,\tilde{\delta}_{x^*})\cap M=\{x^*\} $.
We define:
\begin{gather}
	\delta_{x^*}:=\frac{1}{2}\min\left\lbrace\epsilon_{x^*},\tilde{\delta}_{x^*},d(x^*,x),d(x^*,y),d(x^*,x_{k_0}) \right\rbrace\nonumber\\
	\delta^*_{x^*}:=\frac{1}{2}\min\{\delta_{x^*},d(x^*,\tilde{x}''_{k_0})\}.
\end{gather} 
We observe that $\delta^*_{x^*}>0  $. Indeed, $ d(x^*,\tilde{x}''_{k_0})>0 $ since,
\begin{gather}
	\tilde{x}''_{k_0}\in\partial B_d(x_{k_0},\delta_{k_0})\subset \bar B_d(x_{k_0},\delta_{k_0})\subset B_d(x_{k_0},\delta_{x_{k_0}})\subset B_d(x_{k_0},\frac{1}{2}d(x_{k_0},x^*))\ \text{and}\ x^*\in B_d(x^*,\frac{1}{2}d(x^*,x_{k_0}))\nonumber
\end{gather}
and $B_d(x_{k_0},\frac{1}{2}d(x_{k_0},x^*))\cap B_d(x^*,\frac{1}{2}d(x^*,x_{k_0}))=\emptyset.  $ Therefore, $  d(x^*,\tilde{x}''_{k_0})>0. $ Furthermore, we observe that 
\begin{gather}
	\delta^*_{x^*}\leq\frac{1}{2}\min\{\epsilon_{x^*},d(x^*,\tilde{x}''_{k_0})\}\ \text{and}\ \delta^*_{x^*}\leq\frac{1}{2}\min\{\epsilon_{x^*},d(x^*,y)\}.
\end{gather}
From the Proposition, \ref{main_prop}, we receive that $ \forall\delta\in(0,\delta^*_{x^*}]$:
\begin{gather}
\emptyset\neq K'_{x^*}:=\gamma^*_{x_{k_0},x^*}\left([\tilde{t}_{x_{k_0}},t^*] \right) \cap\partial B_d(x^*,\delta)\ \text{and}\ \emptyset\neq K''_{x^*}:=\gamma\left(t^*,t_y \right)\cap \partial B_d(x^*,\delta) 
\end{gather}
We fix such $ \delta^{(x^*)}\in (0,\delta^*_{x^*}). $ As a result from the last, 
\begin{gather}
	\exists z'_1\in[\tilde{t}_{x_{k_0}},t^*]:\ \gamma^*_{x_{k_0},x^*}(z'_1)\in\partial B_d\left( x^*,\delta^{(x^*)}\right) ,\ \exists z'_2\in[t^*,t_y]:\ \gamma(z''_2)\partial B_d\left( x^*,\delta^{(x^*)}\right) \nonumber.
\end{gather}
Next, we define, 
\begin{gather}
	t_{(x^*)'}:=\inf\left\lbrace t\in[\tilde{t}_{x_{k_0}},t^*]:\ \gamma^*_{x_{k_0},x^*}(t)\in\partial B_d\left( x^*,\delta^{(x^*)}\right) \right\rbrace \nonumber\\
	t_{(x^*)''}:=\sup\left\lbrace t\in[t^*,t_y]:\ \gamma(t)\in\partial B_d\left( x^*,\delta^{(x^*)}\right)  \right\rbrace\nonumber\\
	\text{We set}\ (x^*)':=\gamma^*_{x_{k_0},x^*}\left( t_{(x^*)'}\right) ,\ (x^*)'':=\gamma\left( t_{(x^*)''}\right) \nonumber.
\end{gather}
From the above, we conclude that $t_{(x^*)'},\ t_{(x^*)''}  $ are well defined and by continuity of $ \gamma^*_{x_{k_0},x^*},\ \gamma $ and the fact that $ \partial B_d\left( x^*,\delta^{(x^*)} \right)  $ we receive also that  $ (x^*)',\ (x^*)''\in \partial B_d\left( x^*,\delta^{(x^*)} \right).  $ Furthermore, $ \gamma^*_{x_{k_0},x^*}\left( [\tilde{t}_{x_{k_0}},t_{(x^*)'}]\right) \cap M=\emptyset $, since $ \gamma^*_{x_{k_0},x^*}\left( [\tilde{t}_{x_{k_0}},t_{(x^*)'}]\right) \cap M\subset \gamma^*_{x_{k_0},x^*}\left( [\tilde{t}'_{k_0},t^*]\right) \cap M=\emptyset $ where the last equality holds by the construction property of $ \gamma^*_{x_{k_0},x^*}. $ Moreover, $ \gamma\left( [t_{(x^*)''},t_y]\right)\cap M=\emptyset,$ since $ \gamma\left( [t_{(x^*)''},t_y]\right)\cap M\subset\gamma\left( (t^*,t_y)\right)\cap M=\emptyset$ where the last holds from Claim \ref{cl_0} (iii).

Finally, we define, for the case where $ k^* =\infty $, the path $ \hat{\gamma}:[t_x,t_y]\to U\setminus M $ as follows:
\begin{equation}
	\hat{\gamma}(t):=\begin{cases}
		\gamma(t),\ & t\in[t_x,t_{x'_1}]\nonumber\\
		\tilde{\gamma}_{x'_1,x''_1}(t),\ & t\in[t_{x'_1},t_{x''_1}]\nonumber\\
		\gamma(t),\ & t\in[t_{x''_1},t_{x'_2}]\nonumber\\
		\vdots\      & \vdots\nonumber\\
		\gamma(t),\ & t\in[t_{x''_{k_0-1}},t_{x'_{k_0}}]\nonumber\\
		\tilde{\gamma}_{x'_{k_0},\tilde{x}''_{k_0}}(t),\ & t\in[t_{x'_{k_0}},\tilde{t}_{x_{k_0}}]\nonumber\\
			\gamma^*_{x_{k_0},x^*}(t),\ & t\in[\tilde{t}_{x_{k_0}},t_{(x^*)'}]\nonumber\\
			\tilde{\gamma}_{(x^*)',(x^*)''}(t),\ & t\in[t_{(x^*)'},t_{(x^*)''}]\nonumber\\
			\gamma(t),\ & t\in[t_{(x^*)''},t_y]\nonumber.
	\end{cases}
	\end{equation}

\end{proof}


	\begin{tabular}{l}
		Savvas Andronicou\\ University of Cyprus \\ Department of Mathematics \& Statistics \\ P.O. Box 20537\\
		Nicosia, CY- 1678 CYPRUS
		\\ {\small \tt andronikou.savvas@ucy.ac.cy}
	\end{tabular}
	\begin{tabular}{lr}
		Emmanouil Milakis\\ University of Cyprus \\ Department of Mathematics \& Statistics \\ P.O. Box 20537\\
		Nicosia, CY- 1678 CYPRUS
		\\ {\small \tt emilakis@ucy.ac.cy}
	\end{tabular}
	

\begin{thebibliography}{99999}
	\bibitem{AM24} Andronicou, S., Milakis, E. Estimates on nodal domains of solutions and their applications to free boundary problems, In preparation.
	\bibitem{BBT08} Bruckner, A. M., Thomson, B. S. Classical Real Analysis. (2008), Second Edition
	\bibitem{C52}Caccioppoli, R. Misura e integratione sugli insiemi dimensionalmente orientati, Rend. Accad. Lincei, Cl. Sc. fis. mat. nat., Sirie VIII, fasc 1, 2 (1952), 3-11, 137-146. 
	\bibitem{CR76} Caffarelli, L. A., Riviere, N.M. On the rectifiability of domains with finite perimeter Ann. Scuola Norm. Sup. Pisa 3, no 2 (1976), 177-186.
	\bibitem{DG54} De Giorgi, E. Su una teoria generale dell misura (r-1)-dimensionale in uno spazio ad r dimensioni, Annali do Matematica pura ed applicata, Serie IV, 36 (1954), 191-213. 
	\bibitem{E89} Engelking, R. (1989). General Topology. Sigma series in pure mathematics. Heldermann, 1989.
	\bibitem{F99} Folland, G.B. Real Analysis: Modern Techniques and Their Applications. Pure and Applied Mathematics, Wiley, 2nd edition, 1999.
	\bibitem{G07}Gong, X. H. Connectedness of the solution sets and scalarization for vector equilib-
rium problems, J. Optim. Theory Appl., 133 (2007), 151-161.
         \bibitem{LM11} Lemenant, A., Milakis, E. A stability result for nonlinear Neumann problems in Reifenberg flat domains in  $\R^N$  Publ. Mat. 55 (2011), no. 2, 413-432.
	 \bibitem{LMS13} Lemenant, A., Milakis, E., Spinolo, L. V. Spectral stability estimates for the Dirichlet and Neumann Laplacian in rough domains, J. Funct. Anal. 264 (2013), no. 9, 2097-2135.
	\bibitem{YY23}Li, Y. Y.,  Yan, X. Anisotropic Caffarelli-Kohn-Nirenberg type inequalities, Advances in Mathematics, Volume 419, 2023, 108958
	\bibitem{LL22} Lin, F., Lin, Z. Boundary Harnack principle on nodal domains. Sci. China Math. 65 (2022), 2441-2458.
	\bibitem{T21} Toland, J. F. Path-connectedness in global bifurcation theory, Electron. Res. Arch. 29 (2021), no. 6, 4199-4213.
	\bibitem{ZHW09} Zhong, R., Huang, N. Wong, M. Connectedness and path-connectedness of solution sets to symmetric vector equilibrium problems. Taiw. Journal of Math. Vol. 13, (2009) No. 2B, pp. 821-836.
	\end{thebibliography}
\end{document}